\documentclass[11pt]{amsart}
\usepackage[T1]{fontenc}
\usepackage{amssymb}
\usepackage{bm}
\newtheorem{thm}{Theorem}
\newtheorem{lem}{Lemma}[section]
\newtheorem{prop}{Proposition}[section]
\newtheorem{cor}{Corollary}
\newtheorem{conj}{Conjecture}

\newtheorem{rem}{Remark}[section]
\newtheorem{prob}{Problem}
\usepackage{url}
\usepackage{ulem}
\usepackage[usenames]{color}
\usepackage[utf8]{inputenc}
\numberwithin{equation}{section}

\newcommand{\Z}{\mathbb{Z}}

\DeclareMathOperator{\lcm}{lcm}

\begin{document}

\title[special type of unit equations in two unknowns, I\hspace{-1.2pt}I\hspace{-1.2pt}I]
{Number of solutions to a special type \\of unit equations in two unknowns, I\hspace{-1.2pt}I\hspace{-1.2pt}I}

\author{Takafumi Miyazaki}
\address{Takafumi Miyazaki
\hfill\break\indent Gunma University, Division of Pure and Applied Science,
\hfill\break\indent Graduate School of Science and Technology
\hfill\break\indent Tenjin-cho 1-5-1, Kiryu 376-8515.
\hfill\break\indent Japan}
\email{tmiyazaki@gunma-u.ac.jp}

\author{Istv\'an Pink}
\address{Istv\'an Pink
\hfill\break\indent University of Debrecen, Institute of Mathematics
\hfill\break\indent H-4002 Debrecen, P.O. Box 400.
\hfill\break\indent Hungary}
\email{pinki@science.unideb.hu}

\thanks{The first author was supported by JSPS KAKENHI (No. 20K03553).
The second author was supported by the NKFIH grants 150284 and ANN130909.}

\maketitle

\markleft{Takafumi Miyazaki \& Istv\'an Pink}
\markleft{T. Miyazaki \& I. Pink}
\markright{Special type of unit equations in two unknowns, I\hspace{-1.2pt}I\hspace{-1.2pt}I}

\vspace{-0.3cm}\centerline
{\sl\footnotesize Dedicated to Professor Masaaki Amou on the occasion of his retirement from Gunma University}

\vspace{-0.0cm}
\begin{abstract}
It is conjectured that for any fixed relatively prime positive integers $a,b$ and $c$ all greater than 1 there is at most one solution to the equation $a^x+b^y=c^z$ in positive integers $x,y$ and $z$, except for specific cases. 
We develop the methods in our previous work which rely on a variety from Baker's theory and thoroughly study the conjecture for cases where $c$ is small relative to $a$ or $b$. 
Using restrictions derived from the hypothesis that there is more than one solution to the equation, we obtain a number of finiteness results on the conjecture.
In particular, we find some, presumably infinitely many, new values of $c$ with the property that for each such $c$ the conjecture holds true except for only finitely many pairs of $a$ and $b$. 
Most importantly we prove that if $c=13$ then the equation has at most one solution, except for $(a,b)=(3,10)$ or $(10,3)$ each of which gives exactly two solutions.
Further our study with the help of Schmidt Subspace Theorem among others brings strong contributions to the study of Pillai's type Diophantine equations, notably a general and satisfactory result on a well-known conjecture of M. Bennett on the equation $a^x-b^y=c$ for any fixed positive integers $a,b$ and $c$ with both $a$ and $b$ greater than 1. 
Some conditional results are presented under the $abc$-conjecture as well.
\end{abstract}

\section{Introduction}\label{sec-intro}

The object here is the best possible general estimate of the number of solutions to a special type of the unit equations in two unknowns over the rationals. 
The contents of this paper is regarded as a continuation of a series of our works \cite{MiyPin,MiyPin2}, but many results are independent of it.

First of all, we shall give a brief introduction closely related to a well-known conjecture of S. S. Pillai, which is stated, in a little different way as usual, which we state as follows:

\begin{conj}[Pillai's conjecture] \label{Pillai-conj}
For any fixed positive integer $c,$ there are only finitely many solutions to the equation
\begin{equation} \label{abc-pillai-original}
X^x-Y^y = c
\end{equation}
in positive integers $X,Y,x$ and $y$ all greater than $1.$
\end{conj}

This conjecture has been actively studied to date since Pillai's posing it in 1936.
There are many researchers who contributed to Conjecture \ref{Pillai-conj} in history, many of whom treated the case where the values of some of letters in equation \eqref{abc-pillai-original} are given concretely, and they obtained finiteness results on the solutions or solved the equation completely.
It is perhaps that the most well-known achievement among those works is the one of Mih\u{a}ilescu \cite{Mih} who finally resolved Catalan's conjecture asserting that if $c=1$ then equation \eqref{abc-pillai-original} has only the solution corresponding to the identity $3^2-2^3=1$.
For more detail on these topics, see for example \cite[Ch.\,12]{ShTi}, \cite{Ri-book}, \cite{Wa} or \cite{BiBuMi}.

While Pillai posed his conjecture above, his main interest was the case where the values of base numbers of the two terms in the left-hand side of equation \eqref{abc-pillai-original} are fixed.
Namely, the values of $X,Y$ are given, and $x,y$ are considered as positive integer unknowns.
Then we are led to study the following equation:
\begin{equation} \label{abc-pillai}
a^x-b^y=c
\end{equation}
in positive integers $x$ and $y$, where $a,b$ and $c$ are fixed positive integers with both $a$ and $b$ greater than 1.
The above equation is a very typical example of purely exponential equations, or more generally, of unit equations over the rationals (cf.~\cite[Ch.s 4 to 6]{EvGy}).
Pillai \cite{Pi,Pi2} studied it, finding a sharp estimate of the number of solutions to equation \eqref{abc-pillai}.
In particular, in the case where the base numbers $a$ and $b$ are relatively prime, Pillai proved that the equation does not admit more than one solution whenever $c$ is sufficiently large relative to both $a$ and $b$, in other words, there exists some constant $c_0=c_0(a,b)$ depending only on $a$ and $b$ such that the equation has at most one solution whenever $c>c_0$.
It should be remarked that this result is not effective in the sense that its proof, relying on the theory of rational approximation to algebraic irrationals, does not provide a way to quantify the size of $c_0$ in terms of $a$ and $b$.
This result suggested other researchers to expect that there are only a few solutions to equation \eqref{abc-pillai} in general, and they have attempted to obtain general estimates of the number of solutions for it.
After a famous achievement of Baker on lower bounds for linear forms in logarithms in the late 1960s, it has been mainly applied to exponential Diophantine equations including equation \eqref{abc-pillai}. 
One of the celebrated works in that direction is due to Stroeker and Tijdeman \cite{StTi} who resolved another conjecture of Pillai asserting that $c_{0}(3,2)=13$.
Later, another proof of it was given by Scott \cite{Sc} in 1993.
His method relies on elementary number theory over the ring of the imaginary quadratic field $\mathbb Q(\sqrt{-b^y c}\,)$, which gives a sharp estimate of the number of solutions to equation \eqref{abc-pillai} when in particular $a$ is a prime.
This work is regarded as a breakthrough in the field of purely exponential equations (cf.~\cite[p.242]{Gu}).
Further, Le \cite{Le} found an application of Baker's method under the hypothesis of existence of a few hypothetical solutions to equation \eqref{abc-pillai}, and Bennett \cite{Be_cjm_01} finally refined that idea to prove a general and definitive result as follows:

\begin{prop}[Theorem 1.1 of \cite{Be_cjm_01}] \label{atmost2Pillai}
There are in general at most two solutions to equation $\eqref{abc-pillai}.$
\end{prop}

The proof of this proposition is achieved by the combination of a special type of Baker's method on lower bounds for linear forms in two logarithms together with a gap principle arising from the existence of three hypothetical solutions.
It should be remarked that the case where $\gcd(a,b)>1$ is handled just in an elementary manner, so that such a case is regarded degenerate in a sense.
Proposition \ref{atmost2Pillai} is best possible in the sense that there are a number of examples which allow equation \eqref{abc-pillai} to have two solutions (cf.~\eqref{ex-excep-abc-pillai} below).
After the mentioned achievement of Bennett, several other researchers have attempted to obtain generalizations of Proposition \ref{atmost2Pillai} (see for example \cite{ScSt_jnt2006,ScSt_jtnb2013,ScSt_jtnb2015}).

While Proposition \ref{atmost2Pillai} is sharp, Bennett noted in \cite[(1.2)]{Be_cjm_01} that there are a number of cases where equation \eqref{abc-pillai} has more than one solution, and he found the following examples:
\begin{gather}
2^3-3=2^5-3^3=5,\nonumber\\
2^4-3=2^8-3^5=13,\nonumber\\
2^3-5=2^7-5^3=3,\nonumber\\
3-2=3^2-2^3=1,\nonumber\\
13-3=13^3-3^7=10,\nonumber\\
\label{ex-excep-abc-pillai} 91-2=91^2-2^{13}=89,\\
6-2=6^2-2^5=4,\nonumber\\
6^4-3^4=6^5-3^8=1215,\nonumber\\
15-6=15^2-6^3=9,\nonumber\\
280-5=280^2-5^7=275,\nonumber\\
4930-30=4930^2-30^5=4900.\nonumber
\end{gather}
Further he asked whether all examples for which equation \eqref{abc-pillai} has exactly two solutions come from the identities in \eqref{ex-excep-abc-pillai}, as follows:

\begin{conj}[Conjecture 1.2 of \cite{Be_cjm_01}] \label{atmost1pillai}
Assume that neither $a$ nor $b$ is a perfect power.
Then there is at most one solution to equation $\eqref{abc-pillai},$ except when $(a,b,c)$ belongs to the following set\,$:$
\begin{align}\label{pillai-excep-set}
\{ \,&(2,3,5),(2,3,13),(2,5,3),(3,2,1),\\
&(13,3,10),(91,2,89),(6,2,4),(6,3,1215),\nonumber\\
&(15,6,9),(280,5,275),(4930,30,4900)\,\}.\nonumber
\end{align}
\end{conj}

This conjecture is well-known in the field of purely exponential Diophantine equations, and it is regarded as one of the utmost importance throughout the study of the field (cf.~Conjecture \ref{atmost1} below).
It should be remarked that the case where $\gcd(a,b)=1$ is regarded essential in its study because the remaining one is degenerate in a sense (this situation is similar to that of the generalized Fermat equation (cf.~\cite[Ch.\,14]{Co})).
Based on this, in what follows, we shall consider Conjecture \ref{atmost1pillai} mainly in the non-degenerate case. 

In this paper, we are interested in solving Conjecture $\ref{atmost1pillai}$ when the value of $a$ is fixed, where note that $a$ is the base number of the maximal term in the equation under consideration.
It is noted that the conjecture for prime values of $a$ is reduced to consider case where $\gcd(a,b)=1$ (cf.~\cite[Lemma 4.3]{Be_cjm_01}).
As referred to in \cite{Be_cjm_01}, we emphasize that the conjecture is still difficult under such a restricted situation.
More precisely, we shall consider to work out the next problem:

\begin{prob}\label{a-general-pillai}
Fix the value of $a,$ and prove Conjecture $\ref{atmost1pillai}$ in the following steps\,$:$
\begin{itemize}
\item[Step 1] 
Proving that there is at most one solution to equation \eqref{abc-pillai}$,$ except for only finitely many pairs of $b$ and $c.$
\item[Step 2] 
Finding a way to enumerate all possible pairs of $b$ and $c$ described as exceptional in {\rm Step 1} in an effectively computable finite time. 
\item[Step 3] 
Sieving all pairs of $b$ and $c$ found in {\rm Step 2} completely.
\end{itemize}
\end{prob}

To the best of our knowledge, all values of $a$ not being perfect powers for which Step 1 of the above problem is already worked out are given as in the following with the notation $\nu_p$ denoting the $p$-adic valuation:

\begin{prop}\label{result-exist}
In equation \eqref{abc-pillai}$,$ assume that $\gcd(a,b)=1.$
Let $a$ be any fixed positive integer satisfying at least one of the following conditions\,{\rm :}
\begin{itemize}
\item[$\bullet$] $\max\{2^{\nu_2(a)},3^{\nu_3(a)}\}>\sqrt{a}\,;$
\item[$\bullet$] $a$ is a prime of the form $2^r+1$ with some positive integer $r.$
\end{itemize}
Then there is at most one solution to equation \eqref{abc-pillai}$,$ except for only finitely many pairs of $b$ and $c,$ all of which are effectively determined.
\end{prop}

This proposition is an immediate consequence of \cite[Corollary 2]{MiyPin2} and the proof of \cite[Corollary 1.7]{Be_cjm_01}.
The values of $a$ treated in Proposition \ref{result-exist}, not being perfect powers, are in ascending order as follows:
\[
a=2, 3, 5, 6, 12, 17, 18, 24, 40, 45, 48, 54, 56, 63,\ldots,257,\ldots,65537,\ldots
\]
We note that among the above cases Conjecture \ref{atmost1pillai} is already solved completely for each of the following cases: $a=2$ by Scott \cite[Theorem 6; $p=2$]{Sco}; $a \in \{3,5,17,257,65537,\ldots\}$ (all Fermat primes) by Bennett \cite{Be_cjm_01}. 

Our first main result contributes to Step 1 of Problem \ref{a-general-pillai}.

\begin{thm}\label{a-general-pillai_gcd1_ineff}
In equation \eqref{abc-pillai}$,$ assume that $\gcd(a,b)=1.$
Then, for any fixed $a,$ there is at most one solution to equation \eqref{abc-pillai}$,$ except for only finitely many pairs of $b$ and $c.$
\end{thm}

Although our proof of the above theorem is ineffective in the sense that its proof for each $a$ does not provide a way to quantify the size of finitely many pairs of $b$ and $c$ described `exceptional', the novelty is that Step 1 of Problem \ref{a-general-pillai} in the non-degenerate case is worked out for `arbitrary' values of $a$.

The following are almost immediate corollaries to Theorem \ref{a-general-pillai_gcd1_ineff}, where the latter contributes to answering a question asked by Bugeaud and Luca \cite[Sec.\,7]{BuLu} from the viewpoint of not assuming the primality condition on base numbers.

\begin{cor}\label{cor-a-prime-pillai}
For any fixed prime $a,$ there is at most one solution to equation \eqref{abc-pillai}$,$ except for only finitely many pairs of $b$ and $c.$
\end{cor}

\begin{cor}\label{cor-to-BuLu}
For any fixed positive integer $a$ with $a>1,$ there are only finitely many solutions to the equation
\[
a^{x_1} - a^{x_2} = b^{y_1} - b^{y_2}
\]
in positive integers $b,x_1,x_2,y_1$ and $y_2$ with $\gcd(b,a)=1$ and $x_1 \ne x_2$ such that $a^{x_1}>b^{y_1}.$
\end{cor}

Theorem \ref{a-general-pillai_gcd1_ineff} and its corollaries above are actually consequences of studying a more general equation, which is the main target of this article. 
Namely, in what follows, we consider the equation
\begin{equation} \label{abc}
a^x+b^y=c^z
\end{equation}
in positive integers $x,y$ and $z$, where $a,b$ and $c$ are fixed relatively prime positive integers greater than 1.
The history of the above equation is rich, and it has been actively studied to date since W. Sierpi\'nski's work on the case $(a,b,c)=(3,4,5)$ (cf.~\cite{MiyPin,MiyPin2}).
Below, we shall discuss equation \eqref{abc} in a similar direction to equation \eqref{abc-pillai}, in particular, we are interested in finding sharp estimates of the number of solutions for it.

Recently, the authors \cite{MiyPin} finally established the following definitive result which is regarded as a 3-variable version of Proposition \ref{atmost2Pillai}.

\begin{prop}[Theorem 1 of \cite{MiyPin}] \label{atmost2}
There are in general at most two solutions to equation $\eqref{abc},$ except for $(a,b,c)=(3,5,2)$ or $(5,3,2)$ each of which gives exactly three solutions.
\end{prop}

The proof of the above proposition is achieved by the combination of Baker's method in both complex and $p$-adic cases together with an improved version of the gap principle established by Hu and Le \cite{HuLe,HuLe2,HuLe3} arising from the existence of three hypothetical solutions as well as a certain 2-adic argument relying on the striking result of Scott and Styer \cite{ScoSt_PMD_2016} and many known results on the generalized Fermat equation (for the degenerate case of the proposition see \cite{ScSty}).

Proposition \ref{atmost2} is best possible in the sense that there are infinitely many examples which allow equation \eqref{abc-pillai} to have two solutions.
Indeed, Scott and Styer \cite{ScoSt_PMD_2016} found the following examples, where the condition $\min\{a,b\}=1$ is allowed, by extensive computer search:
\begin{gather}
3^{}+5^{}=2^{3}, \ 3^{3}+5^{}=2^{5}, \ 3^{}+5^{3}=2^{7}; \nonumber\\
3^{}+13^{}=2^{4}, \ 3^{5}+13^{}=2^{8}; \nonumber\\
1^{}+2^{}=3^{}, \ 1^{}+2^{3}=3^{2}; \nonumber\\
2^{2}+5^{}=3^{2}, \ 2^{}+5^{2}=3^{3}; \nonumber\\
2^{}+7^{}=3^{2}, \ 2^{5}+7^{2}=3^{4}; \nonumber\\
2^{3}+3^{}=11^{}, \ 2^{}+3^{2}=11^{}; \nonumber\\ 
3^{}+10^{}=13^{}, \ 3^{7}+10^{}=13^{3}; \nonumber\\
\label{ex-excep-abc} 2^{5}+3^{}=35^{}, \ 2^{3}+3^{3}=35^{};\\ 
2^{}+89^{}=91^{}, \ 2^{13}+89^{}=91^{2};\nonumber\\ 
2^{7}+5^{}=133^{}, \ 2^{3}+5^{3}=133^{}; \nonumber\\ 
2^{8}+3^{}=259^{}, \ 2^{4}+3^{5}=259^{}; \nonumber\\ 
3^{7}+13^{}=2200^{}, \ 3^{}+13^{3}=2200^{}; \nonumber\\ 
2^{13}+91^{}=8283^{}, \ 2^{}+91^{2}=8283^{}; \nonumber\\ 
2^{}+(2^r-1)^{}={2^r+1}^{}, \ 2^{r+2}+(2^r-1)^{2}=(2^r+1)^{2},\nonumber
\end{gather}
where $r$ is any positive integer with $r \ge 2$.
We can observe that the pairs of identities on the lines 6, 8, 10, 11, 12 and 13 in \eqref{ex-excep-abc} are redundant in the sense that each of them is obtained by rearranging another one.
In particular, we emphasize that if we exclude such cases, then the values of $c$ corresponding to one of the remaining pair of identities are $2,3,13,91$ and $2^r+1$ with $r \ge 2$.

Scott and Styer \cite{ScoSt_PMD_2016} put forward the following 3-variable version of Conjecture \ref{atmost1pillai}.

\begin{conj}[\cite{ScoSt_PMD_2016}] \label{atmost1}
Assume that none of $a,b$ and $c$ is a perfect power.
Then there is at most one solution to equation $\eqref{abc},$ except when $(a,b,c)$ or $(b,a,c)$ belongs to the following set\,$:$
\begin{align}\label{excep-set}
\{ \,&(3,5,2),(3,13,2),(2,5,3),(2,7,3),\\
&(2,3,11),(3,10,13),(2,3,35),(2,89,91),\nonumber\\
&(2,5,133),(2,3,259),(3,13,2200),(2,91,8283),\nonumber\\
&(2,2^r-1,2^r+1)\, \},\nonumber
\end{align}
where $r$ is any positive integer with $r=2$ or $r \ge 4.$
\end{conj}

In this paper, as discussed on Conjecture \ref{atmost1pillai}, we are also interested in solving Conjecture \ref{atmost1} when the value of $c$ is fixed, where note that $c$ is the base number of the maximal term in the equation under consideration.
Thus, with the notation $N(a,b,c)$ denoting the number of solutions to equation \eqref{abc}, we are led to consider the following problem.

\begin{prob}\label{prob-abc}
Fix the value of $c,$ and prove Conjecture $\ref{atmost1}$ in the following steps\,$:$
\begin{itemize}
\item[Step 1] 
Proving that $N(a,b,c) \le 1,$ except for only finitely many pairs of $a$ and $b.$
\item[Step 2] 
Finding a way to enumerate all possible pairs of $a$ and $b$ described as exceptional in {\rm Step 1} in an effectively computable finite time. 
\item[Step 3] 
Sieving all pairs of $a$ and $b$ found in {\rm Step 2} completely.
\end{itemize}
\end{prob}

To the best of our knowledge, all values of $c$ not being perfect powers for which Step 1 of the above problem is already worked out are given as in the following:

\begin{prop}\label{result-MiyPin2}
Let $c$ be any fixed positive integer satisfying at least one of the following conditions\,{\rm :}
\begin{itemize}
\item[$\bullet$] $\max\{2^{\nu_2(c)},3^{\nu_3(c)}\}>\sqrt{c}\,;$
\item[$\bullet$] $c$ is a prime of the form $2^r+1$ with some positive integer $r.$
\end{itemize}
Then $N(a,b,c) \le 1,$ except for only finitely many pairs of $a$ and $b,$ all of which are effectively determined.
\end{prop}

This proposition is an immediate consequence of \cite[Corollary 2]{MiyPin2} and the proof of \cite[Theorem 3]{MiyPin2} (cf.~Lemma \ref{c-general-minxy=1--minXY=1} below).
The values of $c$ treated in Proposition \ref{result-MiyPin2}, not being perfect powers, are in ascending order as follows:
\[
c=2, 3, 5, 6, 12, 17, 18, 24, 40, 45, 48, 54, 56, 63,\ldots,257,\ldots,65537,\ldots
\]
We note that among the above cases Conjecture \ref{atmost1} is completely solved for each of the following cases: $c=2$ by Scott \cite[Theorem 6; $p=2$]{Sco}; $c=6$ or $c \in \{3,5,17,257,65537\}$ (Fermat primes found so far) by the authors \cite{MiyPin2}. 

Our second main result provides new values of $c$ for which Step 1 can be worked out.

\begin{thm}\label{c-7etc}
If $c$ is a prime of the form $2^r \cdot 3+1$ with some positive integer $r,$ then $N(a,b,c) \le 1,$ except for only finitely many pairs of $a$ and $b.$
\end{thm}

The values of $c$ treated in the above theorem are presumably infinitely many, and they are in ascending order as follows:
\begin{equation} \label{c-2r31prime}
c=7, 13, 97, 193, 769, 12289, 786433,3221225473, 206158430209,\ldots
\end{equation}
Although Theorem \ref{c-7etc} is ineffective, we give a certain condition on $c$ to make Theorem \ref{c-7etc} effective, and we succeed in finding an effective treatment of only the case where $c=13$.
Our final main result is as follows.

\begin{thm}\label{c13}
If $c=13,$ then Conjecture $\ref{atmost1}$ is true, namely, $N(a,b,13) \le 1,$ except for $(a,b)=(3,10)$ or $(10,3).$
\end{thm}

In the sense of disregarding the redundant cases mentioned just after list \eqref{ex-excep-abc}, the above theorem together with the mentioned earlier works completes the proof of Conjecture \ref{atmost1} for any fixed prime value of $c$ involving the exceptional cases, namely, for $c$ equaling $2,3,13$ or any of (known) Fermat primes.

Theorem \ref{c13} also confirms Conjecture \ref{atmost1pillai} for $a=13$.

\begin{cor}\label{a13}
If $a=13,$ then Conjecture $\ref{atmost1pillai}$ is true, namely, there is at most one solution to equation \eqref{abc-pillai}$,$ except for $(b,c)=(3,10).$
\end{cor}

This corollary together with the mentioned earlier works completes the proof of Conjecture \ref{atmost1pillai} for any fixed prime value of $a$ involving the exceptional set \eqref{pillai-excep-set}, namely, for $a \in \{2,3,13\}$.

The organization of this paper is as follows.
In the next section, we quote several known results based on the theory of Diophantine approximation.
After recalling some ideas to reduce each of our theorems to one of its weak forms in Section \ref{sec-weakform}, under a much more general setting than that of \cite{MiyPin2}, we examine solutions to the system of two equations arising from exceptional triples $(a,b,c)$ with $N(a,b,c)>1$, and find many finiteness results on the solutions when $c$ is small relative to $a$ or $b$.
Here the most important idea is found through applying Baker's method in its non-Archimedean analogue to a certain divisibility relation among solutions, which plays a crucial role in the proof of our results as well as in our previous work \cite{MiyPin2}.
The mentioned finiteness results with the help of Schmidt Subspace Theorem among others lead us to complete the proofs of Theorems \ref{a-general-pillai_gcd1_ineff} and \ref{c-7etc}.
Further, in Section \ref{sec-eff}, after giving some conditions, which involves restricted rational approximation to the square roots of the values in \eqref{c-2r31prime}, sufficient for proving Theorem \ref{c-7etc} effectively, we prove Theorem \ref{c13}, except for a finite search.
We finish half of its finite search in Section \ref{sec-c13-1} by applications of non-Archimedean analogues to Baker's method in number fields of low degrees and a result related to generalized Lebesgue-Nagell equations given recently by Bennett and Siksek \cite{BeSi} together with extensive use of computers, and another half in Section \ref{sec-c13-2} by using a result of Scott \cite{Sco} which restricts the parities of unknowns of the left-hand side of equation \eqref{abc} together with some properties of Lucas sequences and some results on ternary Diophantine equations based on the so-called modular approach.
Section \ref{sec-abc} is devoted to give some applications of $abc$-conjecture to Conjecture \ref{atmost1} in the direction of Problem \ref{prob-abc}.
In the final section we make some remarks towards making Theorems \ref{a-general-pillai_gcd1_ineff} and \ref{c-7etc} effective, with a few open problems for readers.

All computations in this paper were performed by a computer\footnote{Intel Core 7 11800H processor (with 8 cores) and 16GB of RAM} using the computer package MAGMA \cite{BoCaPl}. 
The total computation time did not exceed 1 hour.

\section{Tools from Diophantine approximation}

Henceforth, we frequently use the Vinogradov notation $f \ll_{\kappa_1,\kappa_2,\ldots,\kappa_n} g$ which means that $|f/g|$ is less than some positive constant depending only on $\kappa_1,\kappa_2,\ldots$ and $\kappa_n$, where we simply write $f \ll g$ if the implied constant is absolutely finite.
Note that the implied constant of each Vinogradov notation appearing below is effectively computable unless otherwise referred to.

Let $P[\mathcal A]$ denote the greatest prime factor of a nonzero integer $\mathcal A$, with the convention that $P[1]=P[-1]=1$.

Let $f(t_1,\ldots,t_n)$ be a real-valued function defined on an unbounded region in $\mathbb R^n$.
We write
\[
\lim_{\max\{\,|t_1|,\ldots,|t_n|\,\} \to \infty} f(t_1,\ldots,t_n) = \infty, \ \ \text{effectively},
\]
if, for any $M>0$, there exists some number $T$ which depends only on $M$ and is effectively computable such that $f(t_1,\ldots,t_n)>M$ whenever $\max\{\,|t_1|,\ldots,|t_n|\,\}>T$.

The following is a direct consequence of \cite[Corollary 8.1]{ShTi}. 

\begin{prop}\label{gp}
Let $m$ and $n$ be positive integers with $m>1$ and $n>1$ such that $\max\{m,n\} \ge 3.$
Let $\mathcal X$ and $\mathcal Y$ be variables as a pair of nonzero relatively prime integers.
Then
\[
\lim_{\max\{\,|{\mathcal X}|,\,|{\mathcal Y}|\,\} \to \infty} P[{\mathcal X}^m+{\mathcal Y}^n] = \infty, \ \ \text{effectively}.
\]
\end{prop}

For any positive integer $M>1$, we let $\nu_M(\mathcal A)$ denote the $M$-adic valuation of a nonzero integer $\mathcal A$, namely, the highest exponent $e$ such that $M^e$ divides $\mathcal A$.
Further, for any nonzero rational number $x$, we set $\nu_M(x):=\nu_M(p)-\nu_M(q)$, where $p$ and $q$ are relatively prime integers such that $x=p/q$.

For any algebraic number $\gamma$, we define the absolute logarithmic height ${\rm h}(\gamma)$ of $\gamma$ as follows:
\[
{\rm h}(\gamma) =\frac{1}{[\mathbb{Q}(\gamma):\mathbb{Q}]}\,\Bigl(\, \log c_0 \,+ \,\sum\, \log\, \max \{\,1,\,| \gamma' |\,\} \,\Bigl),
\]
where $c_0>0$ is the leading coefficient of the minimal polynomial of $\gamma$ over $\mathbb Z$, and the sum extends over all conjugates $\gamma'$ of $\gamma$ in the field of complex numbers.

The following is a simple consequence of Bugeaud \cite[Theorem 2]{Bu-madic} with the choice $(\mu,c_2(\mu))=(4,53.6)$.

\begin{prop}\label{Bu-madic}
Let $M$ be an integer with $M>1.$
Let $\alpha_1$ and $\alpha_2$ be nonzero rational numbers such that $\nu_q(\alpha_1)=0$ and $\nu_q(\alpha_2)=0$ for any prime factor $q$ of $M.$
Assume that $\alpha_1$ and $\alpha_2$ are multiplicatively independent.
Let ${\rm g}$ be a positive integer with $\gcd({\rm g},M)=1$ such that
\begin{gather*} \label{Bu-madic-ass1}
\nu_{q}( {\alpha_1}^{{\rm g}}-1 ) \ge \nu_{q}(M), \ \
\nu_{q}( {\alpha_2}^{{\rm g}}-1 ) \ge 1
\end{gather*}
for any prime factor $q$ of $M.$
If $M$ is even, then further assume that
\begin{equation*} \label{Bu-madic-ass2}
\nu_{2}( {\alpha_j}^{{\rm g}}-1 ) \ge 2 \quad (j=1,2).
\end{equation*}
Let $H_1$ and $H_2$ be positive numbers such that
\[
H_j \ge \max \{ \,{\rm h}(\alpha_j),\,\log M\, \} \quad (j=1,2).
\]
Then, for any positive integers $b_1$ and $b_2$ with $\gcd(b_1,b_2,M)=1,$
\[
\nu_M( {\alpha_1}^{b_1} - {\alpha_2}^{b_2}) \le \frac{53.6 \, {\rm g}\, H_1 H_2}{\log^4 M}\, \bigl(\max \{\, \log b^{\ast}+\log \log M+0.64,\,4\log M\, \}\bigl)^2
\]
with $b^{\ast}=b_1/H_2+b_2/H_1.$
\end{prop}

For a number field $\mathbb K$ and a prime ideal $\pi$ in $\mathbb K$, we let $\nu_{\pi}(\alpha)$ denote the exponent of $\pi$ in the prime ideal factorization of the fractional ideal generated by a nonzero element $\alpha$ in $\mathbb K$.

The following is a simple consequence of Bugeaud and Laurent \cite[Th\'eor\`eme 4]{BuLa} with the choice of $(\mu,\nu)=(8,10)$ and $A_j\,(j=1,2)$ chosen so $H_j=(D \log A_j) /f_\pi$.

\begin{prop}\label{BL}
Let $\mathbb K$ be a number field.
Let $\pi$ be a prime ideal in $\mathbb K,$ and $p$ the rational prime lying above $\pi.$
Let $\alpha_1$ and $\alpha_2$ be nonzero elements in $\mathbb K$ such that the fractional ideal generated by $\alpha_1 \alpha_2$ is not divisible by $\pi.$
Assume that $\alpha_1$ and $\alpha_2$ are multiplicatively independent.
Let ${\rm g}$ be a positive integer such that
\[
\nu_{\pi}({\alpha_j}^{\rm g} - 1) \ge 1 \quad (j=1,2).
\]
Let $H_1$ and $H_2$ be positive numbers such that
\[
\quad H_j \ge \max \biggl\{ \frac{D}{f_{\pi}}\,{\rm h}(\alpha_j), \, \log p \biggl\} \quad (j=1,2),
\]
where $D=[\mathbb Q(\alpha_1,\alpha_2):\mathbb Q]$ and $f_{\pi}$ is the inertia index of $\pi.$
Then, for any positive integers $b_1$ and $b_2,$
\begin{multline*}
\nu_{\pi}({\alpha_1}^{b_1}-{\alpha_2}^{b_2}) \le
\frac{27.3 \, D^2 p\, {\rm g} \, H_1 H_2}{{f_{\pi}}^2(p-1)(\log p)^4}\\
\times \Bigl(\max \Bigl\{ \log b^{\ast}+\log \log p+0.4,\,\textstyle{\frac{8 f_{\pi}}{D}}\log p,\,10 \Bigl\} \Bigl)^2
\end{multline*}
with $b^{\ast}=b_1/H_2+b_2/H_1.$
\end{prop}

The following is a particular case of a generalization of Thue-Siegel-Roth theorem due to Ridout \cite{Rid}.

\begin{prop}\label{Ri-thm}
Let $\xi$ be a nonzero algebraic number.
Let $\mathcal S$ be a non-empty set of finitely many primes.
Then, for every positive number $\varepsilon,$ there exists some positive constant $C$ depending only on $\xi,\mathcal S$ and $\varepsilon$ such that the inequality
\[
\left|\xi - \frac{p}{q}\right| > \frac{C}{q^{1+\varepsilon}}
\]
holds for all integers $p$ and $q$ with $q>0$ and $p/q \ne \xi$ such that every prime factor of $q$ belongs to $\mathcal S.$
\end{prop}

For any nonzero $x \in \mathbb Q$, we write $|x|_p=1/p^{\nu_{p}(x)}$ for each prime $p$.
The product formula says that $\prod_{\mu \in \mathbb P \,\cup \,\{\infty\}} |x|_{\mu}=1$, where $|x|_\infty=|x|$ and $\mathbb P$ denotes the set of all prime numbers.

The following is a simplified form of Schmidt Subspace Theorem (cf.~\cite{Sc,Sc2}).

\begin{prop}\label{Sc-subspacethm}
Let $N$ be a positive integer.
Let $\mathcal S$ be a finite set of primes.
For each $\mu \in \mathcal S \cup \{\infty\},$ let $L_{\mu,1},\ldots,L_{\mu,N}$ be $N$ linear forms in $N$ variables with rational coefficients, and assume that they are linearly independent.
Then, for every positive number $\varepsilon$ with $\varepsilon<1,$ there exist finitely many proper subspaces of the $\mathbb Q$ vector space $\mathbb Q^N$ such that any nonzero element $\bm x=(x_1,\ldots,x_N)$ in $\mathbb Q^N$ with $x_i \in \mathbb Z$ for any $i$ satisfying the inequality
\[
\prod_{\mu \in \mathcal S \cup \{\infty\}} \,\prod_{i=1}^{N}\,|L_{\mu,i}(\bm x)|_\mu
\le \frac{1}{(\,\max\{\,|x_i|\,:\, i=1,\ldots,N\,\}\,)^\varepsilon}
\]
belongs to one of those subspaces.
\end{prop}

\section{Reducing to weak forms} \label{sec-weakform}%

For a positive integer $M$, we define $e_M(\mathcal A)$ to be the {\it extended multiplicative order} modulo $M$ of a nonzero integer $\mathcal A$ being coprime to $M$, namely, $e_M(\mathcal A)$ equals the least positive integer $e$ such that $\mathcal A^e$ is congruent to $1$ or $-1$ modulo $M$.
The extended multiplicative orders have properties similar to those of the usual multiplicative orders (cf.~\cite[Lemma 2.1]{MiyPin2}).
In particular, $e_M(\mathcal A)$ is a divisor of $\varphi(M)$.

The following lemma found in \cite[Lemma 2.3]{MiyPin2} tells us that proving Conjecture $\ref{atmost1}$ is reduced to studying a specialization of it.

\begin{lem} \label{weakform}
Let $d$ be a positive divisor of $c$ with $d>2.$
Define positive integers $A$ and $B$ by
\[
A=a^{\,e_d(a)/g}, \quad B=b^{\,e_d(b)/g},
\]
where $g=\gcd(\,e_d(a),e_d(b)\,).$
Then the following hold.
\begin{itemize}
\item[$\bullet$]
The number of solutions $(x,y,z)$ to equation $\eqref{abc}$ equals that to the equation
\[
A^X+B^Y=c^Z
\]
in positive integers $X,Y$ and $Z.$
Further, its correspondence is given by
\[
X=\frac{x}{e_{d}(A)/g}, \ Y=\frac{y}{e_{d}(B)/g}, \ Z=z.
\] 
\item[$\bullet$]
$e_{d}(A)=e_{d}(B).$
\item[$\bullet$]
$(a,b,c)$ belongs to set \eqref{excep-set} if and only if $(A,B,c)$ belongs to the same set.
\end{itemize}
\end{lem}

Based on this lemma, in all situations in the remaining except for the first part of the next section, for some proper divisor $d$ of $c$, in equation \eqref{abc}, we actually consider $(A,B,c,X,Y,z)$ instead of $(a,b,c,x,y,z)$.
In particular, we consider equation \eqref{abc} for the triples $(a,b,c)$ such that $e_{d}(a)=e_{d}(b)$.
It is enough for our purposes to do so, because the following straightforward properties hold.
\begin{gather*}
\max\{a,b\} >\mathcal C \ \Rightarrow \ \max\{A,B\} > \mathcal C\,;\\
X,Y \ll_{c}1 \ \Rightarrow \ x,y \ll_{c}1\,;\\
(A,B,13) \text{ belongs to set \eqref{excep-set}} \ \Leftrightarrow \ \{a,b\}=\{3,10\},
\end{gather*}
where $\mathcal C$ is any positive number.
This restriction is basically not essential in the study of Conjecture $\ref{atmost1}$, but it will turn out that it is helpful in the proof of Theorem \ref{c13} (cf.~Section \ref{sec-c13-2}). 

\section{Cases for general values of $c$}\label{sec-c-general}

The numbers $\mathcal C_i\,(i=1,2,\ldots,5)$ and ${\mathcal C_5}'$ appearing in this or the next section are some positive constants which depend only on $c$ and are effectively computable.

In this section, we give several finiteness results on the solutions to equation \eqref{abc} for general values of $c$.
On equation \eqref{abc}, it is clear that the implication
\[
z \ll_c 1 \ \Rightarrow \ \max\{a,b\} \ll_c 1
\]
holds, and this will be frequently and implicitly used below.

\begin{lem}\label{c-general-coprime}
If $\max\{a,b\}>\mathcal C_1,$ then $\gcd(x,y,c)=1$ for any solution $(x,y,z)$ to equation $\eqref{abc}.$
\end{lem}

\begin{proof}
Suppose that $\gcd(x,y,c)>1$ for some solution $(x,y,z)$.
Then there is some prime factor of $c$ dividing both $x$ and $y$, say $p$.
It suffices to show that $z \ll_c 1$ or $a,b \ll_c 1$.
Note that $p \ll_{c} 1$.

Write $x=p x_0,\,y=p y_0$.
If $p=2$, then $c$ is even, so that both $a,b$ are odd, leading to $c^z=(a^{x_0})^2+(b^{y_0})^2 \equiv 2 \pmod{4}$, whereby $z=1$.
For odd $p$, one considers the greatest prime factor of each of both sides of equation \eqref{abc} to see that
\[
P[{\mathcal X}^m+{\mathcal Y}^n]=P[c],
\]
where $(m,n):=(p,p)$ and $({\mathcal X},{\mathcal Y}):=(a^{x_0},b^{y_0})$.
For each possible $p$, we shall apply Proposition \ref{gp}.
Since $P[c] \ll_c 1$, it turns out that $\mathcal X,\mathcal Y \ll_{c} 1$, so that $a,b \ll_c 1$.
\end{proof}

Since Conjecture \ref{atmost1} is true for $c=2$ by Scott \cite[Theorem 6;\,$p=2$]{Sco} (or \cite[Corollary 1]{MiyPin2}), we may assume that $c$ is not a power of 2.
Thus there is some divisor $c_1$ of $c$ with $c_1>2$ such that
\begin{equation}\label{c-cond}
\gcd(\,c_1,\varphi(c_1)\,)=1.
\end{equation}
Note that $c_1$ is odd.

\begin{lem}\label{1st-app}
Let $(x,y,z)$ be a solution to equation $\eqref{abc}$ such that $\gcd(x,y,c_1)=1.$
Then
\[
z \ll \max\{1,F_1\} \cdot e \,\log a \,\log b,
\]
where
\[
F_1=F_1(c,c_1)=\max\biggr\{1,\,\frac{\log^2 (\log c\,\log c_1)}{ \log^2 {c_1} }
\biggr\} \cdot \frac{1}{\nu_{c_1}(c)\,\log^2 {c_1}},
\]
and $e=e_{c_1}(a) \cdot e_{c_1}(b) \cdot \lcm(\,e_{c_1}(a),e_{c_1}(b)\,).$
\end{lem}

\begin{proof}
By equation \eqref{abc},
\[
\nu_{c_1}(a^{2x}-b^{2y}) \ge \nu_{c_1}(a^x+b^y) = \nu_{c_1}(c^z) \ge \nu_{c_1}(c)\cdot z.
\]
To find an upper bound for the leftmost hand side above, we shall apply Proposition \ref{Bu-madic} for
\[
M:=c_1, \ (\alpha_1,\alpha_2):=(a,b), \ (b_1,b_2):=(2x,2y).
\]
Since, for each $h \in \{a,b\}$, one finds that $h^{2\,e_{c_1}(h)} \equiv (h^{e_{c_1}(h)})^2 \equiv (\pm1)^2 \equiv 1 \pmod{c_1}$, and that $\gcd(\,2\,e_{c_1}(h),c_1\,)=\gcd(\,e_{c_1}(h),c_1\,)=1$ by \eqref{c-cond} with $e_{c_1}(h) \mid \varphi(c_1)$, one may set ${\rm g}:=2\,\lcm(\,e_{c_1}(a),e_{c_1}(b)\,)$.
Observe that $\gcd(b_1,b_2,M)=\gcd(2x,2y,c_1)=\gcd(x,y,c_1)=1$ by assumption.
Putting
\[
H_1:=\log a', \ H_2:=\log b',
\]
where $a'=\max\{a,c_1\}$ and $b'=\max\{b,c_1\}$, one obtains
\[
\nu_{c_1}(a^{2x}-b^{2y}) \ll \frac{\lcm(\,e_{c_1}(a),e_{c_1}(b)\,)\,\log a'\,\log b'}{\log^4 {c_1}} \cdot \mathcal B^2,
\]
where
\[
\mathcal B=\max \biggr\{ \log \biggr( \frac{2x}{\log b'}+\frac{2y}{\log a'} \biggr)+\log \log c_1, \ \log c_1 \biggr\}.
\]
By the trivial inequalities $x<\frac{\log c}{\log a}\,z,\, y<\frac{\log c}{\log b}\,z$ from equation \eqref{abc}, observe that
\[
\log \biggr( \frac{2x}{\log b'}+\frac{2y}{\log a'} \biggr) 
\le \log \biggr( \frac{2x}{\log b}+\frac{2y}{\log a} \biggr)
<\log \biggr( \frac{4\log c}{\log a\,\log b}\,z \biggr).
\]
This together with the two bounds for $\nu_{c_1}(a^{2x}-b^{2y})$ implies
\[
T \ll \frac{1}{\nu_{c_1}(c)} \cdot \frac{\log a'}{\log a} \cdot \frac{\log b'}{\log b} \cdot \frac{\lcm(\,e_{c_1}(a),e_{c_1}(b)\,)}{\log^4 {c_1}} \cdot {\mathcal B'}^2,
\]
where
\[
T:=\frac{4\,z}{\,\log a\,\log b\,}, \quad \mathcal B':=\log \,\max \bigr\{ \,(\log c\,\log c_1)\,T ,\,c_1\,\bigr\}.
\]
For each $h \in \{a,b\}$, since ${c_1}<h^{2\,e_{c_1}(h)}$ as $h^{2\,e_{c_1}(h)} \equiv 1 \pmod{c_1}$, it follows that
\[
\frac{\log h'}{\log h} = \max\biggl\{1,\,\frac{\log c_1}{\log h}\biggr\} \le 2\,e_{c_1}(h),
\]
so that
\[
T \ll \frac{e}{\nu_{c_1}(c)\,\log^4 {c_1}} \cdot {\mathcal B'}^2.
\]
If $(\log c\,\log c_1)\,T \le c_1$, then $\mathcal B' =\log c_1$, so that
\begin{equation}\label{T-ub}
T \ll \frac{e}{\nu_{c_1}(c)\,\log^2 {c_1}}.
\end{equation}
Finally, suppose that $(\log c\,\log c_1)\,T > c_1$.
Then
\[
T \ll \frac{e}{\nu_{c_1}(c)\,\log^4 {c_1}} \cdot \log^2 \bigr(\,(\log c\,\log c_1)\,T\,\bigr).
\]
If $\log c\,\log c_1>T$, then
\begin{equation}\label{T-ub2}
T \ll \frac{\log^2 (\log c\,\log c_1)}{ \log^2{c_1} }\cdot \frac{e}{\nu_{c_1}(c)\,\log^2 {c_1}}.
\end{equation}
While if $\log c\,\log c_1 \le T$, then putting 
\[
R:=\frac{e}{\nu_{c_1}(c)\,\log^4 {c_1}},
\]
one has
\begin{equation}\label{ineq-R}
\quad
\frac{T}{\log^2 T} <u\,R
\end{equation}
for some constant $u$ with $1 \le u \ll 1$. 
If $R<3$, then inequality \eqref{ineq-R} implies that 
\begin{equation}\label{T-ub3}
T \ll 1.
\end{equation}
In the remaining parts, we assume that $R \ge 3$.
If $T \ge 4u R \log^2 R\,(>\exp(2))$, then replacing the number $T$ in \eqref{ineq-R} by $4u R \log^2 R$ implies that
\[
\frac{R}{\log^2 R} <4u,
\]
so that $R \ll 1$, leading to the same as \eqref{T-ub3} by \eqref{ineq-R}.
If $T < 4u R \log^2 R$, then
\[
T \ll R \log^2 R=\frac{e}{\nu_{c_1}(c)\,\log^4 {c_1}} \cdot \log^2 \biggr( \frac{e}{\nu_{c_1}(c)\,\log^4 {c_1}} \biggr).
\]
Since $e \le \varphi(c_1)^4<{c_1}^4$, the above implies the same as \eqref{T-ub}.
The combination of \eqref{T-ub}, \eqref{T-ub2}, \eqref{T-ub3} completes the proof.
\end{proof}

From now on, suppose that equation \eqref{abc} has two solutions, say $(x,y,z)$ and $(X,Y,Z)$, with $(x,y,z) \ne (X,Y,Z)$.
We have
\begin{eqnarray*}
&a^x+b^y=c^z,\\
&a^X+b^Y=c^Z.
\end{eqnarray*}
There is no loss of generality in assuming that $z \le Z$.
The above two equations are referred to the 1st and 2nd equation, respectively.

In what follows, we put
\[
\ \ \Delta:=|xY-Xy| \ \ (>0).
\]
Note that $\Delta$ is nonzero in general (cf.~\cite[Lemma 3.3]{HuLe}).

It holds trivially that
\begin{equation}\label{trivial-ineqs}
x<\frac{\log c}{\log a}\,z, \ \ y<\frac{\log c}{\log b}\,z, \ \ X<\frac{\log c}{\log a}\,Z, \ \ Y<\frac{\log c}{\log b}\,Z.
\end{equation}

Based on the contents of the previous section, from now on, we assume that $e_{c_1}(a)=e_{c_1}(b)$ (however, this can allow $a$ or $b$ to be a perfect power).
We let $E_1$ be its common value, namely,
\[
E_1:= e_{c_1}(a)=e_{c_1}(b).
\]
Recall that there are some properties on $E_1$ similar to those of the usual multiplicative orders modulo $c_1$, in particular, $E_1$ is a divisor of $\varphi(c_1)$.

For each $h \in \{a,b\}$, since $c_1>2$, we can uniquely define $\delta_h \in \{1,-1\}$ by the following congruence:
\begin{equation}\label{cong-h}
h^{E_1} \equiv \delta_h \mod{c_1}.
\end{equation}

The following lemma is a well-known fact in $p$-adic calculations.

\begin{lem}\label{padic-lemma}
Let $p$ be a prime number.
Let $U$ and $V$ be relatively prime nonzero integers.
Assume that
\[\begin{cases}
\, U \equiv V \pmod{p} & \text{if $p \ne 2$},\\
\, U \equiv V \pmod{4} & \text{if $p=2$}.
\end{cases}\]
Then
\[
\nu_p(U^N-V^N)=\nu_p(U-V)+\nu_p(N)
\]
for any positive integer $N.$
\end{lem}

\begin{lem}\label{div}
The following hold.
\begin{itemize}
\item[\rm (i)]
$h^{\Delta} \equiv \epsilon \pmod{c^z},$ where $\epsilon=(-1)^{y+Y}$ or $(-1)^{x+X}$ according to whether $h=a$ or $b.$
\item[\rm (ii)]
$\Delta \equiv 0 \pmod{E_1}.$
\item[\rm (iii)]
$\gcd(a^{E_1}-\delta_a,b^{E_1}-\delta_b) \cdot \Delta/E_1 \equiv 0 \pmod{{c_1}^z}.$
\end{itemize}
\end{lem}

\begin{proof}
Let $h \in \{a,b\}$.\par
(i) One reduces 1st, 2nd equations modulo $c^z$, and combines the resulting congruences to obtain the assertion (cf.~proof of \cite[Lemma 3.2]{MiyPin2}). \par
(ii) Assertion (i) in particular says that $h^{\Delta} \equiv \pm 1 \pmod{c_1}$.
This implies the asserted divisibility relation by a property of the extended multiplicative orders (cf.~\cite[Lemma 2.1\,(i)]{MiyPin2}). \par
(iii) By (ii) and its proof, one knows that $h^{\Delta} \equiv \epsilon \pmod{c_1}$ with $\Delta \equiv 0 \pmod{E_1}$. 
Since this congruence on the modulus $c_1$ ($>\!2$) is also obtained by raising both sides of congruence \eqref{cong-h} to the $\Delta/E_1$-th power, it turns out that $\epsilon={\delta_h}^{\Delta/E_1}$.
It follows from (i) that
\begin{equation}\label{cong-h^Delta}
h^{\Delta} \equiv {\delta_h}^{\Delta/E_1} \mod{{c_1}^z}.
\end{equation}

Let $p$ be any prime factor of $c_1$.
Note that $p$ is odd.
It follows from \eqref{cong-h^Delta} that
\[
\nu_p(h^{\Delta}-{\delta_h}^{\Delta/E_1}) \ge \nu_p(c_1) \cdot z.
\]
From congruence \eqref{cong-h} observe that $h^{E_1} \equiv \delta_h \pmod{p}$. 
To calculate the left-hand side of the above displayed inequality, Lemma \ref{padic-lemma} is applied for $(U,V)=(h^{E_1},\delta_h)$ and $N=\Delta/E_1$.
It turns out that
\[
\nu_p(h^{E_1}-\delta_h)+\nu_p(\Delta/E_1) \ge \nu_p(c_1) \cdot z.
\]
Since $p$ is chosen arbitrarily, it follows that $(h^{E_1}-\delta_h) \cdot \Delta/E_1$ is divisible by ${c_1}^z$, and the assertion holds.
\end{proof}

\begin{lem}\label{Delta-ub}
If $\max\{a,b\}>\mathcal C_1,$ then
\[
\Delta \ll F_2\,z,
\]
where $F_2=F_2(c,c_1,E_1)=\max\{1,F_1\} \cdot{E_1}^3 \log^2 c.$
\end{lem}

\begin{proof}
Since $\Delta=|x Y-X y|<\max\{x Y,X y\}$, it follows from inequalities \eqref{trivial-ineqs} that
\begin{equation} \label{Delta-basic-ub}
\Delta<\frac{\log^2 c}{\log a\,\log b}\,z\,Z.
\end{equation}
This together with the upper bound for $Z$ obtained from Lemmas \ref{c-general-coprime} and \ref{1st-app} on the solution $(X,Y,Z)$
implies the assertion.
\end{proof}

The above lemma in particular says that the magnitude of $\Delta$ as a function of $z$ is `very small' relative to that of ${c_1}^z$ appearing in the divisibility relation of Lemma \ref{div}\,(iii).
This observation is the most important one in our previous work \cite{MiyPin2}, and also for the present one.
Based on this, we put
\[
\mathcal D_1:=\frac{{c_1}^z}{\Delta'},
\]
where 
\[
\Delta':=\gcd(\Delta/E_1,{c_1}^z).
\]
It turns out from Lemma \ref{div}\,(iii) that
\begin{equation} \label{cong-h-calD}
h^{E_1} \equiv \delta_h \mod{\mathcal D_1}
\end{equation}
for each $h \in \{a,b\}$.

In what follows, we frequently use the following inequalities:
\[
c_1 \ll_{c} 1, \ \ E_1 \ll_{c} 1.
\]

\begin{lem}\label{calD-l}
If $\max\{a,b\}>\mathcal C_2,$ then $\log \mathcal D_1 \gg (\log c_1)\,z.$
\end{lem}

\begin{proof}
We may assume that $\max\{a,b\}>\mathcal C_1$.
Lemma \ref{Delta-ub} gives that $\Delta' \le \Delta/E_1 \ll (F_2/E_1) z$, so that
\[
\Delta' <\eta \,z
\]
for some $\eta>0$ which depends only on $c$ and is effectively computable.
Then
\[
\log \mathcal D_1 = z \log c_1 - \log \Delta'
>z \log c_1 - \log (\eta z),
\]
so that
\[
\frac{\log \mathcal D_1}{z} > \log c_1 - \frac{\log (\eta z)}{z} \gg \log c_1,
\]
whenever $z$ exceeds some constant depending only on $\eta$.
This completes the proof.
\end{proof}

The following lemma is the most important step in the proofs of our results.

\begin{lem}\label{2nd-app}
If $\max\{a,b\}>\mathcal C_3,$ then
\[
z\,Z \ll \frac{1}{\nu_{c_1}(c) \log^2 c_1} \, {E_1}^3\,\log a\,\log b.
\]
\end{lem}

\begin{proof}
The proof proceeds similarly to that of Lemma \ref{1st-app}.
We may assume that $\max\{a,b\}>\mathcal C_1,\mathcal C_2$.
Note that $\mathcal D_1>1$ by Lemma \ref{calD-l}.

First, from 2nd equation, observe that
\[
\nu_{\mathcal D_1}(a^{2X}-b^{2Y}) \ge \nu_{\mathcal D_1}(a^X+b^Y)=\nu_{\mathcal D_1}(c^Z) \ge \nu_{{c_1}^z}(c^Z) \ge \nu_{c_1}(c) \cdot \lfloor Z/z \rfloor.
\]
To find an upper bound for the leftmost hand side above, we shall apply Proposition \ref{Bu-madic} with the same values of $\alpha_1,\alpha_2,b_1,b_2$ as those in the proof of Lemma \ref{1st-app} (on the solution $(X,Y,Z)$).
In this case, we set $M:=\mathcal D_1$, and all required conditions for this choice are satisfied by condition \eqref{c-cond}, congruence \eqref{cong-h-calD} and Lemma \ref{c-general-coprime}.
Then
\[
\nu_{\mathcal D_1}(a^{2X}-b^{2Y}) \ll \frac{E_1 \log a'\, \log b'}{\log^4 \mathcal D_1} \cdot \mathcal B^2,
\]
where $a'=\max\{a,\mathcal D_1\},\, b'=\max\{b,\mathcal D_1\}$ and
\[
\mathcal B=\max \biggr\{ \log \biggr( \frac{2X}{\log b'}+\frac{2Y}{\log a'} \biggr)+\log \log \mathcal D_1, \ \log \mathcal D_1 \biggr\}.
\]
Since $\lfloor Z/z \rfloor \gg Z/z$, and
\[
\log \biggr( \frac{2X}{\log b'}+\frac{2Y}{\log a'} \biggr)<\log \biggr( \frac{4\log c}{\log a\,\log b}\,Z \biggr),
\]
the two bounds for $\nu_{\mathcal D_1}(a^{2X}-b^{2Y})$ together imply
\[
T \ll \frac{1}{\nu_{c_1}(c)} \cdot \frac{\log a'}{\log a} \cdot \frac{\log b'}{\log b} \cdot \frac{E_1\,z^2}{\log^4 \mathcal D_1} \cdot {\mathcal B'}^2,
\]
where
\[
T:=\frac{4\,z\,Z}{\,\log a\,\log b\,}, \quad \mathcal B':=\log\,\max \bigr\{\, (\log c)(\log \mathcal D_1)\,T/z, \,\mathcal D_1\,\bigr\}.
\]
Note that $\log \mathcal D_1 \gg (\log c_1)\,z$ by Lemma \ref{calD-l}.
Since $\mathcal D_1<a^{2E_1},b^{2E_1}$ by congruence \eqref{cong-h-calD}, one has
\[
T \ll \frac{{E_1}^3 z^2}{\nu_{c_1}(c)\,\log^4 \mathcal D_1} \cdot {\mathcal B'}^2.
\]
If $(\log c)(\log \mathcal D_1)\,T/z \le \mathcal D_1$, then ${\mathcal B}'=\log {\mathcal D}_1$, so that 
\[
T \ll \frac{{E_1}^3 z^2}{\nu_{c_1}(c)\,\log^2 \mathcal D_1} \ll \frac{{E_1}^3}{\nu_{c_1}(c) \log^2 c_1},
\]
which gives the asserted inequality.
Finally, suppose that $(\log c)(\log \mathcal D_1)\,T/z >\mathcal D_1$.
Then
\[
\frac{\mathcal D_1}{(\log c)\log \mathcal D_1} < T/z \ll \frac{{E_1}^3\,z}{\nu_{c_1}(c)\,\log^4 \mathcal D_1} \cdot \log^2\bigr(\,(\log c)(\log \mathcal D_1)\,T/z \,\big).
\]
If $(\log c)\log \mathcal D_1>T/z$, then
\[
\frac{\mathcal D_1}{\log \mathcal D_1} \ll_c T/z \ll_c \log \mathcal D_1,
\]
so that $\mathcal D_1 \ll_c 1$, whereby $z \ll \log \mathcal D_1 \ll_c 1$.
While if $(\log c)\log \mathcal D_1 \le T/z$, then
\[
\frac{\mathcal D_1}{\log \mathcal D_1} \ll_c T/z \ll_c \frac{z}{\log^4 \mathcal D_1} \cdot \log^2(T/z).
\]
Since $z \ll \log \mathcal D_1$, it is not hard to see from the right-hand inequality above that $T/z \ll_c 1$, so that the left-hand one yields $\mathcal D_1 \ll_c 1$, and $z \ll_c 1$.
This completes the proof.
\end{proof}

\begin{lem}\label{c-general-xyXY-ub}
If $\max\{a,b\}>\mathcal C_3,$ then
\[
\max\{x,y,X,Y\} \ll {F_3}^2{E_1}^6 \ll_{c} 1,
\]
where $F_3=F_3(c,c_1)=\frac{1}{\nu_{c_1}(c)} \frac{\log^2 c}{\log^2 c_1}.$
\end{lem}

\begin{proof}
Lemma \ref{2nd-app} together with inequalities \eqref{trivial-ineqs} yields
\begin{align*}
z\,Z \ll F_3 {E_1}^3 \cdot \frac{\log a}{\log c} \cdot \frac{\log b}{\log c}
&<F_3 {E_1}^3 \cdot \min \biggr\{\dfrac{z}{x},\dfrac{Z}{X} \biggr\} \cdot \min \biggr\{\dfrac{z}{y},\dfrac{Z}{Y} \biggr\}\\
& \le F_3 {E_1}^3 \cdot \min \biggr\{\dfrac{z^2}{xy},\dfrac{Z^2}{XY} \biggr\}.
\end{align*}
This implies that
\[
X Y \ll F_3 {E_1}^3 \cdot Z/z \ll F_3 {E_1}^3 \cdot F_3 {E_1}^3/(x y),
\]
leading to the assertion.
\end{proof}

\begin{rem}\rm
(i) Readers should notice that without assuming $e_{c_1}(a)=e_{c_1}(b),$ Lemma \ref{c-general-xyXY-ub} in particular says that all possible values of $x,y,X$ and $Y$ are only finitely many and `effectively' determined, whenever the value of $c$ is fixed.
This fact itself is highly non-trivial, and it will bring us very strong contributions to weak forms of Conjecture \ref{atmost1} which are already interesting (cf.~Theorem \ref{a-general-pillai_gcd1_ineff} and Theorem \ref{x>1y>1z>1} below).
\\
(ii) Since relation \eqref{c-cond} holds if $c_1$ is an odd prime, one can always choose $c_1$ as the least odd prime factor of $c$.
Thus, only applying $p$-adic analogue to Baker's method instead of $m$-adic one due to Bugeaud (Proposition \ref{Bu-madic}) is enough to prove Lemma \ref{c-general-xyXY-ub}. 
On the other hand, one should rely on Bugeaud's mentioned work to find a better estimate of the implied constants in that lemma.
\end{rem}

A bright idea of Luca appearing in the proof of \cite[Lemma 7]{Luca_ActaArith_12} over the ring of Gaussian integers is used in the proof of the following lemma.

\begin{lem}\label{c-general-minxy=1--minXY=1}
If $\max\{a,b\}>\mathcal C_4,$ then the following hold.
\begin{itemize}
\item[\rm (i)] $\min\{x,y\}=1.$
\item[\rm (ii)] $\min\{X,Y\}=1.$
\end{itemize}
\end{lem}

\begin{proof}
By Lemma \ref{c-general-xyXY-ub}, one finds that $x,y,X,Y \ll_c 1$ and $Z \ll_c z$.
\par
(i) Similarly to the proof of Lemma \ref{c-general-coprime}, for each possible pair $(x,y)$, we shall apply Proposition \ref{gp} with $(m,n):=(x,y),\, (\mathcal X,\mathcal Y):=(a,b)$.
It turns out from 1st equation that $a,b \ll_c 1$, unless either $\min\{x,y\}=1$ or $(x,y)=(2,2)$.

Suppose that $x=y=2$.
We shall observe that this leads to $z \ll_c 1$.
Note that $\Delta\,=2|X-Y|$ is a positive even integer.
As seen in the proof of Lemma \ref{c-general-coprime}, it suffices to consider when $c$ is odd.
Below we shall argue over the ring of Gaussian integers.
A usual factorization argument on 1st equation over $\Z[i]$ shows that there is some element $\beta$ with $\beta\bar{\beta}=c$ such that $a+b i$ equals $\beta^z$ times a unit.
Note that $\beta$ is not a unit and is coprime to $\bar{\beta}$.
This implies that
\[
a=\dfrac{\epsilon}{2}\,(\,\beta^z +\delta \,{\bar{\beta}}^z\,)
\]
for some unit $\epsilon$ and $\delta=\pm 1$.
Observe that $a^{2\Delta} \equiv 1 \pmod{c^z}$ from Lemma \ref{div}\,(i), and that
\[
a^4=\dfrac{1}{2^4}\,(\,\beta^z +\delta\,{\bar{\beta}}^z\,)^4.
\]
These together yield
\[
(\,\beta^z+\delta\,{\bar{\beta}}^z\,)^{2\Delta} \equiv 2^{2\Delta} \mod{c^z}.
\]
One reduces this modulo $\bar{\beta}^z$ to obtain $\beta^{2\Delta z} \equiv 2^{2\Delta} \pmod{\bar{\beta}^z}$.
Taking a prime element $\pi$ dividing $\bar{\beta}$, one has
\[
\nu_{\pi} (\beta^{2 \Delta z}-2^{2\Delta}) \ge z.
\]
To find an upper bound for the left-hand side above, we shall apply Proposition \ref{BL} for $(\alpha_1,\alpha_2):=(\beta,2),\,(b_1,b_2):=(2\Delta z,2\Delta)$.
Note that the corresponding $p$ is some odd prime factor $(\ge 5)$ of $c$ and $(D, f_{\pi})=(2,1)$.
One may set ${\rm g}:=\pi\bar{\pi}-1=p-1$.
Further, one may set $H_1:=\log c$ as ${\rm h}(\beta)=\frac{1}{2}\log c$, and $H_2:=\log p$.
Therefore,
\[
\nu_{\pi} (\beta^{2\Delta z}-2^{2\Delta}) \ll \frac{(p-1) \log c}{\log^3 p} \,\log^2 \max \{\mathcal B, p\},
\]
where
\[ \quad
\mathcal B= \Delta z \cdot \left( 2 + \frac{2\log p}{z \log c} \right).
\]
To sum up, the two bounds for $\nu_{\pi} (\beta^{2\Delta z}-2^{2\Delta})$ together imply
\[
z \ll_c \,\log^2 (2\Delta z).
\]
Since $\Delta=2|X-Y| \ll_c 1$, the above implies that $z \ll_c 1$.\par
(ii) Similarly to the proof of (i), Proposition \ref{gp} is applied for 2nd equation, and it suffices to show $(X,Y) \ne (2,2)$.
In the case where $X=Y=2$, from 2nd equation, one can find that 
\[
z \ll_{c} \,\log^2 (2\Delta Z)
\]
similarly to the proof of (i).
Since $Z \ll_{c} z$, the above implies that $z \ll_c 1$.
This completes the proof.
\end{proof}

The above lemma brings the following consequence to (a more general version of) Terai's conjecture (cf.~\cite{Te}) involving the condition $\min\{x,y,z\}>1$.
His conjecture is a special case of Conjecture \ref{atmost1} and it asserts that if there is some solution to equation \eqref{abc} in positive integers $x,y$ and $z$ all greater than $1,$ then $N(a,b,c) \le 1,$ except for specific cases.

\begin{thm}\label{x>1y>1z>1}
For any fixed $c,$ if there is some solution to equation \eqref{abc} in positive integers $x$ and $y$ both greater than $1,$ then $N(a,b,c) \le 1,$ except for only finitely many pairs of $a$ and $b,$ all of which are effectively determined.
\end{thm}

\begin{rem}\rm
(i) As seen around Conjecture \ref{atmost1}, there are infinitely many cases for which equation \eqref{abc} has some solution in positive integers $x$ and $y$ both greater than $1$ but $N(a,b,c)=2$.
\\
(ii) Assuming the truth of $abc$-conjecture, it is possible to show that a conjecture of Le (cf.~\cite{Le2}), which is a particular case of the mentioned conjecture of Terai, holds true except for only finitely many cases, namely, there is at most one solution to equation \eqref{abc} in positive integers $x,y$ and $z$ all greater than $1,$ except for only finitely many triples of $a,b$ and $c$ (cf.~Lemma \ref{abc-xneX-yneY}\,(ii) below).
\end{rem}

The following proposition is a particular case of \cite[Theorem 4.2]{BuLu}.

\begin{prop} \label{BuLu-newyork}
Let $m$ and $n$ be positive integers with $m>n.$
Let $q$ be a positive integer with $q>1.$
Then there are only finitely many solutions to the equation
\begin{equation} \label{mnq}
\mathcal Z^m - \mathcal Z^n= q^{y_1} - q^{y_2}
\end{equation}
in positive integers $\mathcal Z, y_1$ and $y_2$ with $\mathcal Z>1$ and $\gcd(\mathcal Z,q)=1.$
\end{prop}

The proof of the above proposition largely relies on Propositions \ref{Ri-thm} and \ref{Sc-subspacethm}.
Their ineffective nature makes the proof of Lemma \ref{c-prime-xneXandyneY} below provide no effective way to quantify the size of $\mathcal C$ in terms of $c$.
This means that the proofs of Theorems \ref{a-general-pillai_gcd1_ineff} and \ref{c-7etc} ultimately rely on Subspace Theorem.
It is worth noting that Subspace Theorem is required only for the case where $m \ge 3,\,n=1$ and $q$ is composite.

For the sake of completeness, and for later discussions on the possibility to make the results introduced in Section 1 except Theorem \ref{c13} effective (cf.~Section \ref{sec-appro}), at the end of this section we give a proof of Proposition \ref{BuLu-newyork} following \cite{Lu,BuLu}, with some modifications and simpler treatments.

\begin{lem}\label{c-prime-xneXandyneY}
There exists some constant $\mathcal C$ depending only on $c$ such that if $\max\{a,b\}>\mathcal C,$ then $x \ne X$ and $y \ne Y.$
\end{lem}

\begin{proof}
Lemma \ref{c-general-xyXY-ub} says that $x,y,X,Y \ll_c 1$.
If $x=X$, then eliminating the terms $a^x$ from 1st, 2nd equations leads to
\[
b^Y - b^y = c^Z - c^z.
\]
For each possible pair $(y,Y)$, Proposition \ref{BuLu-newyork} is applied for $(m,n,q):=(Y,y,c), (\mathcal Z,y_1,y_2):=(b,Z,z)$ to see in particular that $z$ has to be less than some constant depending only on $c$.
Similarly the same conclusion holds when $y=Y$.
This completes the proof.
\end{proof}

We are now in position to prove Theorem \ref{a-general-pillai_gcd1_ineff}.

\begin{proof}[Proof of Theorem $\ref{a-general-pillai_gcd1_ineff}$]
Let $a>1$ be a fixed positive integer. 
For some pair $(b,c)$ with $\gcd(a,b)=1$, suppose that equation \eqref{abc-pillai} has two solutions.
Write
\[
c+b^y=a^x, \quad c+b^Y=a^X.
\]
Since there is no loss of generality in assuming by \cite[Theorem 6]{LeV} that $c>1$, Lemma \ref{c-prime-xneXandyneY} is applied to the triple $(c,b,a)$ to imply that both $b,c$ have to be less than some constant depending only on $a$.
This completes the proof.
\end{proof}

The following lemma will not be actually used in the proofs of the results of this paper (cf.~Section \ref{sec-op-1}).

\begin{lem}\label{c-general-xleEyleE}
If $\max\{a,b\}>\mathcal C_5,$ then
\[
\max\{x,y\} \le \biggl \lfloor \frac{E_1 \log c}{\log c_1}\biggl \rfloor.
\]
\end{lem}

\begin{proof}
Set $m:=\min\{a,b\}$.
By Lemma \ref{div}\,(iii), it follows that ${c_1}^z$ divides $W$, where $W=(m^{E_1}-\delta_m) \cdot \Delta/E_1.$
Note that $W \ge {c_1}^z$.
We shall estimate the size of $W$.
Since there is no loss of generality in assuming by Lemma \ref{c-general-xyXY-ub} that $\Delta <xyXY \ll {F_3}^2 {E_1}^6$, it follows that
\[
W \le (m^{E_1}+1) \cdot \Delta/E_1 \ll m^{E_1} \cdot {F_3}^2 {E_1}^5.
\]
Since $m=\min\{a,b\}<c^{\,z/\max\{x,y\}}$ by the first two inequalities in \eqref{trivial-ineqs}, one has
\[
W \ll \big(c^{ \,E_1/\max\{x,y\} }\bigr)^z \cdot {F_3}^2 {E_1}^5.
\]
Since ${c_1}^z \le W$, it follows that
\begin{equation} \label{c-general-maxxy-ineq}
{L_1}^z \ll {F_3}^2 {E_1}^5 \ll_c 1,
\end{equation}
where $L_1=c_1/\,c^{\,E_1/\max\{x,y\}}$.

Put $l=\frac{E_1\log c}{\log c_1}$.
Then $L_1>1$ if and only if $\max\{x,y\}>\lfloor l \rfloor$.
Indeed,
\[
L_1>1 \,\Leftrightarrow\, c_1>c^{\,E_1/\max\{x,y\}} \,\Leftrightarrow\, \max\{x,y\}> l.
\]
It follows from \eqref{c-general-maxxy-ineq} that if $\max\{x,y\}>\lfloor l \rfloor$, then
\[
z \ll_c \frac{1}{\log L_1} \le \frac{1}{\log (\,c_1\,/\,c^{\,E_1/(\lfloor l \rfloor+1)}\,)} \ll_c 1.
\]
This completes the proof.
\end{proof}

A summary of this section is as follows.

\begin{lem}\label{c-general-summary}
Assume that $c$ is not a power of $2.$
Let $c'$ be a positive divisor of $c$ with $c'>2$ such that $\gcd(\,c',\varphi(c')\,)=1.$
Assume that $e_{c'}(a)=e_{c'}(b).$
Put $E=e_{c'}(a).$
Then the following hold.
\begin{itemize}
\item[$\bullet$]
There exists some constant $K_1$ which depends only on $c$ and is effectively computable such that if $\max\{a,b\}>K_1,$ then
\begin{alignat*}{3}
&\min\{x,y\}=1, &&\max\{x,y\} \le \displaystyle \biggl \lfloor \frac{E \log c}{\log c'}\biggl \rfloor, \\
&\min\{X,Y\}=1, &\quad &\max\{X,Y\}<\,\frac{\mathcal C_0\,E^6 \log^4 c}{\nu_{c'}(c)^2 (\log c')^4\max\{x,y\}},
\end{alignat*}
where $\mathcal C_0$ is some positive absolute constant\footnote{We can take $\mathcal C_0=4.78 \cdot 10^7$ if $c$ is an odd prime.} being effectively computable.
\item[$\bullet$]
There exists some constant $K_2$ depending only on $c$ such that if $\max\{a,b\}>K_2,$ then $(x,Y)=(1,1)$ or $(y,X)=(1,1).$
\end{itemize}
\end{lem}

We finish this section by introducing a proof of Proposition \ref{BuLu-newyork}, where the implied constant of each Vinogradov notation depends only on the parameters of equation \eqref{mnq} (namely, $m,n$ and $q$) and is not always effectively computable.

\begin{proof}[Proof of Proposition $\ref{BuLu-newyork}$]
Let $(\mathcal Z, y_1,y_2)$ be a solution to equation \eqref{mnq}.
Clearly, $y_1>y_2$.
We show that $\mathcal Z, y_1 \ll 1$ in any case.
Note that $\mathcal Z, y_1 \ll 1$ if $\mathcal Z \ll 1$ or $y_1 \ll 1$.

Firstly, observe that
\begin{align} 
q^{y_1/m} \ll \mathcal Z \ll q^{y_1/m},\label{dominant}\\
q^{y_2} \mid (\mathcal Z^{m-n}-1).\label{qy2div}
\end{align}

Define an integer $M$ by
\[
M:=q^{y_1} - \mathcal Z^m.
\]
Note that $M$ is nonzero as $\mathcal Z>1$ and $\gcd(\mathcal Z,q)=1$ by assumption.
We are interested in the size of $M$.
Since $M= q^{y_2} - \mathcal Z^n$ by equation \eqref{mnq}, it holds trivially that
\[
|M|<\max\{q^{y_2}, \mathcal Z^n\}.
\]
On the other hand,
\begin{align*}
|M|&=|(q^{y_1/m})^m - \mathcal Z^m|\\
&=|q^{y_1/m} - \mathcal Z| \cdot (\,(q^{y_1/m})^{m-1}+ (q^{y_1/m})^{m-2}\mathcal Z +\cdots
+\mathcal Z^{m-1}\,)\\
&\ge |q^{y_1/m} - \mathcal Z| \cdot m\,\min\{q^{y_1/m},\mathcal Z\}^{m-1}.
\end{align*}
These two bounds for $M$ together lead to
\[
|q^{y_1/m} - \mathcal Z| <\frac{\max\{q^{y_2}, \mathcal Z^n\}}{m\, \min\{q^{y_1/m},\mathcal Z\}^{m-1}}.
\]
Note that the left-hand side above does not vanish.

Write
\[
y_1=m Q_1+R_1
\]
for integers $Q_1,R_1$ with $Q_1 \ge 0$ and $0 \le R_1<m$.
Dividing both sides of the previous displayed inequality by $q^{Q_1}$ gives
\[
\left|q^{R_1/m} - \frac{\mathcal Z}{q^{Q_1}}\right| <\frac{\max\{q^{y_2}, \mathcal Z^n\}}{m\,q^{Q_1} \min\{q^{y_1/m},\mathcal Z\}^{m-1}}.
\]
Since $q^{y_1/m} \ll \mathcal Z \ll q^{y_1/m}$ by \eqref{dominant}, it follows that
\begin{equation} \label{approx}
\left|\xi - \frac{\mathcal Z}{q^{Q_1}}\right| \ll \frac{\max\{q^{y_2}, q^{n\,y_1/m}\}} {q^{Q_1+(m-1) y_1/m}},
\end{equation}
where $\xi$ is an algebraic number given by $\xi=q^{R_1/m}$.

Since $\mathcal Z \ll q^{y_1/m}$, and $q^{y_2} \ll \mathcal Z^{m-n}$ by \eqref{qy2div}, one has
\[
q^{y_2} \ll q^{(m-n) y_1/m}.
\]
It follows from
\eqref{approx} that
\begin{equation} \label{approx-2}
\begin{split}
\left|\xi - \frac{\mathcal Z}{q^{Q_1}}\right| &
\ll \frac{q^{\,\max\{m-n,\,n\}\,y_1/m }} {q^{Q_1+(m-1)y_1/m}}=\frac{1} {q^{Q_1+Q_2}},
\end{split}
\end{equation}
where
\begin{align*}
Q_2&=(\,m-1-\max\{m-n,n\}\,) \cdot y_1/m\\
&=\min\{n-1,m-n-1\} \cdot (Q_1+R_1/m).
\end{align*}

Let $\varepsilon:=Q_2 / Q_1$.
Then \eqref{approx-2} gives
\[
\left|\xi - \frac{\mathcal Z}{q^{Q_1}}\right|
\ll \frac{1} {(q^{Q_1})^{1+\varepsilon}}.
\]
This is a restricted rational approximation to $\xi$.
If $\varepsilon \gg 1$, then Proposition \ref{Ri-thm} with $\mathcal S$ the set of all prime factors of $q$ is applied to see that
\[
\left|\xi - \frac{\mathcal Z}{q^{Q_1}}\right|
> \frac{C}{(q^{Q_1})^{1+\varepsilon/2}}
\]
for some constant $C$ depending only on $\xi,\mathcal S$ and $\varepsilon$.
The two bounds for $|\xi - \mathcal Z/q^{Q_1}|$ together imply that $q^{Q_1} \ll 1$, in particular, $y_1 \ll m Q_1 \ll 1$.
By these observations, we shall next estimate the size of $\varepsilon$.
Observe that
\begin{align*}
\varepsilon=Q_2 / Q_1&=\min\{n-1,\,m-n-1\} \cdot (\,1+R_1/(m Q_1)\,)\\
& \ge \min\{n-1,\,m-n-1\}.
\end{align*}
Therefore, it suffices to consider when $n=1$ or $m-n=1$.

Suppose that $n=m-1$.
It follows from \cite[Lemma 9.2\,(i)]{MiyPin} that $\mathcal Z \gg q^{3y_2}$ if $q^{y_2}>n$.
Thus, in any case, 
\begin{equation} \label{qy2small}
q^{y_2} \ll \mathcal Z^{1/3}.
\end{equation}
Multiplying both sides of equation \eqref{mnq} by $m^m$ leads to
\begin{equation} \label{mnqmm}
m^m q^{y_1} - (m \mathcal Z-1)^m = m^m q^{y_2} - G(\mathcal Z),
\end{equation}
where $G$ is a polynomial with integer coefficients of degree $m-2$ given by
\begin{align*}
G(\mathcal Z)&=(m \mathcal Z-1)^m - m^m(\mathcal Z^m-\mathcal Z^{m-1})\\
&=\binom{m}{2}(m \mathcal Z)^{m-2} - \binom{m}{3}(m \mathcal Z)^{m-3} + \cdots + (-1)^m.
\end{align*}
Note that the left-hand side of \eqref{mnqmm} does not vanish (as it is not divisible by $m$).
Similarly to a previous case, putting $\xi:=m q^{R_1/m}$, one finds from \eqref{mnqmm} with \eqref{qy2small} that
\begin{align*}
\left| \xi - \frac{m \mathcal Z -1}{q^{Q_1}}\right| &< \frac{\max\{m^m q^{y_2},G(\mathcal Z)\}} {m\,q^{Q_1} \min\{m q^{y_1/m},m\mathcal Z-1\}^{m-1}}\\
&\ll \frac{\max\{q^{y_2},\mathcal Z^{m-2}\}
}{
q^{Q_1} \min\{q^{y_1/m},\mathcal Z\}^{m-1}
}\\
&\ll \frac{
\mathcal Z^{\,\max\{1/3,\,m-2\}}
}{
q^{Q_1} (q^{y_1/m})^{m-1}
}\\
&\ll \frac{q^{\,\max\{1/3,\,m-2\} \, y_1/m}}{
q^{Q_1+(m-1)y_1/m}} = \frac{1}{q^{Q_1+Q_2}},
\end{align*}
where $Q_2= \min\{m-4/3,1\} \cdot y_1/m$.
Since $m \ge 2$ and $y_1 / m \ge Q_1$, it follows that $Q_2 \ge (2/3)Q_1$, so that
\[
\left|\xi - \frac{m \mathcal Z -1}{q^{Q_1}}\right| \ll \frac{1} {(q^{Q_1})^{5/3}}.
\]
Similarly to a previous case, by considering the above restricted rational approximation to $\xi$, it turns out from applying Proposition \ref{Ri-thm} that $Q_1 \ll 1$, and $y_1 \ll 1$.
Therefore, it remains to consider when $n=1$ with $m \ge 3$.

In the remaining part, we suppose that $m \ge 3$ and $n=1$.
By divisibility relation \eqref{qy2div}, 
\[
q^{y_2} \ \Big| \ (\mathcal Z-1) \cdot \frac{\mathcal Z^{m-1}-1}{\mathcal Z-1}.
\]
Since the greatest common divisor of $\mathcal Z-1$ and $\frac{\mathcal Z^{m-1}-1}{\mathcal Z-1}$ is a divisor of $m-1$, the above relation implies that
\begin{equation} \label{q1q2uv}
\mathcal Z-1={q_2}^{y_2}\,u, \ \ \frac{\mathcal Z^{m-1}-1}{\mathcal Z-1}={q_3}^{y_2}\,v
\end{equation}
for some positive integers $q_2,q_3$ with $q_2 q_3=q$, and some positive rational numbers $u,v$ each of whose denominator is a divisor of $m-1$.
Let $\delta>0$ be any fixed constant satisfying $\delta<\frac{2(m-2)}{(m-1)(m-2)+2}\,(<1)$.
We distinguish two cases according to whether $\max\{u,v\} \ge \mathcal Z^\delta$ or not.

Assume that
\begin{equation} \label{uorvnotsmall}
u \ge \mathcal Z^\delta \text{ \ or \ } v \ge \mathcal Z^\delta.
\end{equation}
Then
\[
q^{y_2} = {q_2}^{y_2} \cdot {q_3}^{y_2} = \frac{\mathcal Z^{m-1}-1}{u v}
< \frac{\mathcal Z^{m-1}}{\max\{u,v\}}
\le \mathcal Z^{m-1-\delta}.
\]
This leads to
\[
| M |=| \mathcal Z^1 - q^{y_2} | \ll \mathcal Z^{\,\max\{1,\,m-1-\delta\}}=\mathcal Z^{m-1-\delta}\ll (q^{y_1/m})^{m-1-\delta}.
\]
Similarly to previous cases, the above implies that
\[
\left|q^{R_1/m} - \frac{\mathcal Z}{q^{Q_1}} \right| 
\ll \frac{(q^{y_1/m})^{m-1-\delta}}{q^{Q_1+(m-1)\,y_1/m}}=\frac{1}{q^{Q_1+\delta(y_1/m)}}
\le \frac{1}{(q^{Q_1})^{1+\delta}}.
\]
Thus, Proposition \ref{Ri-thm} is applied for the above restricted rational approximation to $q^{R_1/m}$ to find that the solutions to equation \eqref{mnq} with property \eqref{uorvnotsmall} are only finitely many.

From now on, we consider only the solutions with the property that
\begin{equation} \label{uandvsmall}
u< \mathcal Z^\delta \text{ \ and \ } v<\mathcal Z^\delta.
\end{equation}
We shall observe that $\mathcal Z \ll 1$.
Write $D$ for the least common multiple of the denominators of $u$ and $v$.
We know that $D$ is composed of primes less than $m$.
Observe that
\[
\frac{\mathcal Z^{m-1}-1}{\mathcal Z-1}=
m-1+\sum_{k=2}^{m-1}\binom{m-1}{k}(\mathcal Z-1)^{k-1}.
\]
With the notation in \eqref{q1q2uv}, one multiplies both sides above by $D^{m-2}$ and rearranges terms to find that
\begin{equation} \label{key-subspace}
(m-1) D^{m-2} = {q_3}^{y_2}v D^{m-2} - \sum_{k=2}^{m-1}\binom{m-1}{k}\, {q_2}^{(k-1)y_2}\,u^{k-1} D^{m-2}.
\end{equation}
For applying Proposition \ref{Sc-subspacethm} to the above relation, set $N:=m-1\,(\ge 2)$, and let $\mathcal S$ be the set of all primes which divide $q$ or are less than $m$.
For each $p \in \mathcal S$, define $N$ linear forms $L_{p,i}=L_{p,i}(X_1,...,X_N) \,(i=1,...,N)$ by $L_{p,i}:=X_i$ for any $i$.
Further, define $N$ linear forms $L_{\infty,i} =L_{\infty,i}(X_1,...,X_N) \,(i=1,...,N)$ by $L_{\infty,i}:=X_i$ for $i>1$, and
\[
L_{\infty,1}:= X_1 - \sum_{k=2}^{N}\binom{N}{k} X_k.
\]
Note that the defined forms depend only on $N$, and that $L_{\mu,1},...,L_{\mu,N}$ are linearly independent for each $\mu \in \mathcal S \cup\{\infty\}$.

We shall estimate the left-hand side of the inequality in Proposition \ref{Sc-subspacethm} at the points $\bm x=(x_1,...,x_N) \in {\mathbb Q}^N$ given by
\begin{align*}
&x_1:={q_3}^{y_2} v D^{N-1} \ =\frac{\mathcal Z^N-1}{\mathcal Z-1}\,D^{N-1}, \\
&x_k:= {q_2}^{(k-1)y_2}\,u^{k-1} D^{N-1} \ =(\mathcal Z-1)^{k-1} D^{N-1} \ \text{for $k \ge 2$}.
\end{align*}
Clearly, $x_i \in \mathbb Z$ for any $i$, and $\max\{|x_i|:i=1,...,N\}=x_1 \ll \mathcal Z^{N-1}$.
Note that $q_2,q_3,D$ are composed of primes in $\mathcal S$.
By the product formula, one finds that
\begin{align*}
\prod_{\mu \in \mathcal S} \, |L_{\mu,1}(\bm x)|_\mu
=\prod_{\mu \in \mathcal S} \, |x_1|_\mu
&=\prod_{\mu \in \mathcal S} \, |{q_3}^{y_2}v D^{N-1}|_\mu\\
&=\prod_{\mu \in \mathcal S} \, |{q_3}^{y_2}D^{N-2}|_\mu\,\cdot\, \prod_{\mu \in \mathcal S} \, |v D|_\mu\\
&=\frac{1}{|{q_3}^{y_2}D^{N-2}|_\infty}\cdot\prod_{\mu \in \mathcal S} \, |v D|_\mu
\le \frac{1}{{q_3}^{y_2}D^{N-2}},
\end{align*}
and $|L_{\infty,1}(\bm x)|_\infty= N D^{N-1}$ by \eqref{key-subspace}.
For $k \ge 2$,
\begin{align*}
&\prod_{\mu \in \mathcal S \cup\{\infty\}} |L_{\mu,k}(\bm x)|_\mu
=\prod_{\mu \in \mathcal S \cup\{\infty\}} |x_k|_\mu
=\prod_{\mu \in \mathcal S \cup\{\infty\}} \,\bigr|\,{q_2}^{(k-1)y_2}\,u^{k-1} D^{N-1}\,\bigr|_{\mu}\\
&=
\prod_{\mu \in \mathcal S \cup\{\infty\}} |\,{\rm num}\,(u)^{k-1}\,|_\mu \ \cdot
\prod_{\mu \in \mathcal S \cup\{\infty\}} \bigr|\,{q_2}^{(k-1)y_2}\,D^{N-1}/{\rm den}\,(u)^{k-1}\,\bigr|_\mu\\
&=\Biggr(\prod_{\mu \in \mathcal S \cup\{\infty\}} |{\rm num}\,(u)|_\mu\Biggr)^{k-1} \cdot \,1\\
&\le \bigl(\,|{\rm num}\,(u)|_{\infty}\,\bigl)^{k-1}
= {\rm num}\,(u)^{k-1},
\end{align*}
where ${\rm num}\,(u),{\rm den}\,(u)$ are the numerator and denominator of $u$, respectively.
Since ${\rm den}\,(u) \ll 1$, and ${q_3}^{y_2}\,v \gg \mathcal Z^{N-1}$ by \eqref{q1q2uv}, it follows from property \eqref{uandvsmall} that
\begin{alignat*}{3}
\prod_{\mu \in \mathcal S \cup\{\infty\}}\,\prod_{i=1}^{N} \, |L_{\mu,i}(\bm x)|_\mu &
\le \frac{{\rm num}\,(u)^{1+2+\cdots+N-1} \cdot N D^{N-1}}{{q_3}^{y_2} D^{N-2}}&\\
&\ll \frac{ {\rm num}\,(u)^{\frac{N(N-1)}{2}} }{{q_3}^{y_2}}&\\
& \ll \frac{ u^{\frac{N(N-1)}{2}}\,v}{ \mathcal Z^{N-1} }&\\
& \ll \frac{ \mathcal Z^{(N^2/2-N/2+1)\delta} }{ \mathcal Z^{N-1} }
=\frac{ 1 }{ \mathcal Z^{N-1-(N^2/2-N/2+1)\delta} }
=\frac{ 1 }{ (\mathcal Z^{N-1})^{\varepsilon} },
\end{alignat*}
where $\varepsilon=1-\frac{N(N-1)+2}{2(N-1)}\,\delta$.
To sum up, 
\[
\prod_{\mu \in \mathcal S \cup\{\infty\}} \,\prod_{i=1}^{N} \, |L_{\mu,i}(\bm x)|_\mu
\ll \frac{1}{ (\,\max\{\,|x_i|:i=1,...,N\,\}\,)^{\varepsilon} },
\]
Since $0<\varepsilon<1$ by the choice of $\delta$, Proposition \ref{Sc-subspacethm} is applied to imply that there are only finitely many proper subspaces of $\mathbb Q^N$ such that each of the points $\bm x$ under consideration belongs to one of them.
This means that for each such $\bm x=(x_1,...,x_N)$ the linear relation $\sum_{k=1}^{N}R_k x_k =0$ holds for some $N$-tuple $(R_1,...,R_N)$ of integers belonging to a finite subset of ${\mathbb Z}^N \backslash \{\bm 0\}$ depending only on $N$.
Thus
\[
R_1 w_1 + \sum_{k=2}^{N} R_k {w_2}^{k-1} =0,
\]
where $w_1=\frac{{\mathcal Z}^N-1}{\mathcal Z-1}, \,w_2=\mathcal Z-1$.
Observe that $w_2 \mid R_1 N$ as $w_2 \mid R_1 w_1$ with $\gcd(w_1,w_2) \mid N$.
If $R_1 \ne 0$, then $w_2 \le R_1 N \ll 1$.
If $R_1=0$, then the non-trivial algebraic equation $\sum_{k=2}^{N} R_k {w_2}^{k-1}=0$ for $w_2$ holds, in particular, $w_2 \ll_{R_2,...,R_N} 1 \ll 1$.
To sum up, $\mathcal Z=w_2+1 \ll 1$, and the solutions to equation \eqref{mnq} with property \eqref{uandvsmall} are only finitely many.
This completes the proof.
\end{proof}

\section{Cases for prime values of $c$}\label{sec-c-prime}

Throughout this section, we assume that $c$ is an odd prime, and prove stronger results than some of the lemmas in the previous section.
In particular, Lemma \ref{c-general-xleEyleE} is substantially improved.
Note that in \eqref{c-cond} we can take $c_1$ as $c$ itself.
Further, we simply write $E_1=E$.

\begin{lem}\label{c-prime-coprime}
Assume that $c$ is an odd prime.
Then $\gcd(x,y,c)=1$ for any solution $(x,y,z)$ to equation $\eqref{abc}.$
\end{lem}

\begin{proof}
This can be proved in the same way as in the proof of \cite[Lemma 8.1(i)]{MiyPin2}.
\end{proof}

From now on, similarly to the previous section, we assume that equation \eqref{abc} has two solutions $(x,y,z)$ and $(X,Y,Z)$ with $z \le Z$.

The following is a direct consequence of Scott \cite[Lemma 6]{Sco}.

\begin{prop}\label{twoclass}
If $c$ is an odd prime, then $x \not\equiv X \pmod{2}$ or $y \not\equiv Y \pmod{2},$ except for $(a,b,c)=(3,10,13)$ or $(10,3,13).$
\end{prop}

\begin{lem}\label{c-prime-parity}
Assume that $c$ is an odd prime.
If $\max\{a,b\}>\mathcal C_4,$ where $\mathcal C_4$ is defined in Lemma $\ref{c-general-minxy=1--minXY=1},$ then $\Delta$ is odd.
In particular, $\max\{a,b\} \le \mathcal C_4$ if $E$ is even.
\end{lem}

\begin{proof}
Suppose that $\Delta$ is even, that is,
\[
x Y \equiv X y \mod{2}.
\]
On the other hand, by Proposition \ref{twoclass}, one may assume that
\[
x \not\equiv X \mod{2} \quad \text{or} \quad y \not\equiv Y \mod{2}.
\]
These congruence conditions together imply that $x,y$ are even or $X,Y$ are even.
Thus Lemma \ref{c-general-minxy=1--minXY=1} gives $a,b \le \mathcal C_4$.
The second assertion follows from Lemma \ref{div}\,(ii).
\end{proof}

The following lemma for the case where $E=3$ will be mainly used in the proofs of Theorems \ref{c-7etc} and \ref{c13}.

\begin{lem}\label{c-prime-xleE-2--yleE-2}
Assume that $c$ is an odd prime.
Further, assume that $E$ is odd with $E>1.$
If $\max\{a,b\}>{\mathcal C_5}',$ then $\max\{x,y\} \le E-2.$
\end{lem}

\begin{proof}
The proof proceeds along the similar lines to that of Lemma \ref{c-general-xleEyleE}.
By Lemma \ref{div}\,(iii), it follows that $c^z$ divides $(m^E-\delta_m) \cdot \Delta/E$, where $m=\min\{a,b\}$.
Since $c$ is a prime, $m \not\equiv \delta_m \pmod{c}$ as $E>1$, and $E$ is odd, one finds that $c^z$ divides $W$, where
\[
W:=\frac{m^E-\delta_m}{m-\delta_m} \cdot \gcd(\Delta/E,c^z).
\]

We shall estimate the size of $W$.
Since one may assume by Lemma \ref{c-general-xyXY-ub} that $\Delta<xyXY \ll E^6$, it follows that
\[
W \le (m^{E-1}+m^{E-2}+\cdots+1) \cdot \Delta/E \ll m^{E-1} \cdot E^5.
\]
Since $m<c^{\,z/\max\{x,y\}}$, it follows that
\begin{equation}\label{W}
W \ll c^{\,z(E-1)/\max\{x,y\}} \cdot E^5.
\end{equation}

If $\max\{x,y\} \ge E$, then, since $c^z \le W \ll c^{\,z(E-1)/E} \cdot E^5$ by \eqref{W}, one has $c^{z/E} \ll E^5$, so that $c^z \ll_{c} 1$, whereby $z \ll_c 1$.

Finally, suppose that $\max\{x,y\}=E-1$.
It suffices to show that $z \ll_{c} 1$.
By \eqref{W},
\[ W/c^z \ll E^5 \]
with $W \equiv 0 \pmod{c^z}$.
Thus
\[ W =K \cdot c^z \]
for some positive integer $K$ satisfying $K \ll E^5 \ll_{c} 1$.
Since $\gcd(\Delta/E,c^z)=c^n$ with some integer $n \ge 0$ being absolutely finite and effectively computable as $\Delta \ll E^6<c^6$, one concludes from the definition of $W$ that
\begin{equation}\label{hKzn}
\frac{m^E-\delta_m}{m-\delta_m}=K \cdot c^{z-n}.
\end{equation}
This implies that the integer
\[ \ \ \
m^{E-1}+\delta_m m^{E-2}+{\delta_m}^2\,m^{E-3}+\cdots +{\delta_m}^{E-1} \left(=\frac{m^E-\delta_m}{m-\delta_m}\right)
\]
is composed of only the primes dividing $K$ and $c$.
In particular,
\[
P[\,g(m)\,] \ll_c 1,
\]
where $g$ is a polynomial with integer coefficients given by $g(t)=t^{E-1}+\delta_m t^{E-2}+{\delta_m}^{2}\,t^{E-3}+\cdots +{\delta_m}^{E-1}$.
Since $g$ has no multiple roots of degree $E-1 \ge 2$, it is known for each possible $E$ that 
\[
\lim_{j \to \infty} P[\,g(j)\,]=\infty, \ \ \text{effectively}
\]
(cf.~\cite[Theorem 7.1]{ShTi}).
Therefore, $m \ll_{c,E} 1$.
Since $z-n \ll_{m,E} 1$ by \eqref{hKzn}, one has $z \ll_{c} 1$.
\end{proof}

A summary of this section is as follows.

\begin{lem}\label{c-prime-summary}
Assume that $c$ is an odd prime and that $e_{c}(a)=e_{c}(b).$
Put $E=e_{c}(a).$
Assume that $E>1.$
Then there exist some constants $L_1$ and $L_2$ which depend only on $c$ and are effectively computable such that the following hold.
\begin{itemize}
\item[$\bullet$]
If $E$ is even, then $\max\{a,b\}<L_1.$
\item[$\bullet$]
If $\max\{a,b\}>L_2$ and $E$ is odd, then $\max\{x,y\} \le E-2.$
\end{itemize}
\end{lem}

\section{Proof of Theorem \ref{c-7etc}}

Let $c$ be any prime of the form $c=2^r \cdot 3 +1$ with a positive integer $r$.
Assume that $\max\{a,b\}>\mathcal C_4,{\mathcal C_5}'$.
Suppose that equation \eqref{abc} has two solutions $(x,y,z),(X,Y,Z)$ with $z \le Z$.
Note that $E>1$ by \cite[Corollary 1]{MiyPin2}.
Since $E \mid \Delta$ by Lemma \ref{div}\,(ii), and $\Delta$ is odd by Lemma \ref{c-prime-parity}, it follows that $E$ is odd.
Now, Lemmas \ref{c-prime-xleE-2--yleE-2} and \ref{c-general-minxy=1--minXY=1}\,(ii) say that
\begin{gather*}
x \le E-2 \text{ and } y \le E-2;\\
X=1 \text{ or } Y=1,
\end{gather*}
respectively.
Further, one can show that $E=3$ as follows.
Recall that $E \mid \varphi(c)$.
Since $\varphi(c)=c-1=2^r \cdot 3$, it follows that $E=3$.
Therefore, $x=y=1$.
Lemma \ref{c-prime-xneXandyneY} says that $a,b \le \mathcal C$.
This completes the proof of Theorem \ref{c-7etc}.

\section{Towards effective version of Theorem \ref{c-7etc}} \label{sec-eff}

Based on the proof of Theorem \ref{c-7etc} with Proposition \ref{twoclass}, we can find the following proposition towards making Theorem \ref{c-7etc} effective.

\begin{prop}\label{prop-c7etc}
Assume that $c$ is a positive integer having the following three properties\,$:$
\begin{itemize}
\item[\rm (i)] $c$ is a prime\,$;$
\item[\rm (ii)] $\varphi(c)=2^r \cdot 3$ with some positive integer $r\,;$
\item[\rm (iii)] There exist a positive constant $\epsilon$ with $\epsilon<0.6$ and a positive integer $k_0$ such that the inequality
\[
\left|\sqrt{c}-\frac{P}{Q}\right| >\frac{1}{Q^{1+\epsilon}}
\]
holds for any integer $P$ and for any positive integer $Q$ of the form $Q=c^k$ with some positive integer $k$ satisfying $k \ge k_0.$
\end{itemize}
Let $b$ be a positive integer with $b>1$ and $\gcd(b,c)=1.$
Let $Y_0$ be a positive even integer with $4 \le Y_0 \ll_c 1.$
Then the equation
\begin{equation}\label{bcY0zZ}
c^Z-c^z=b^{Y_0}-b
\end{equation}
has no solution in positive integers $z$ and $Z$ such that $c^z>b,$ whenever $b$ exceeds some constant which depends only on $c,\epsilon$ and $k_0$ and is effectively computable.
\end{prop}

As seen in \eqref{c-2r31prime} there are many values of $c$ for which both conditions (i) and (ii) of the above proposition hold.
Among those values, we could only find that condition (iii) also holds for $c=13$ with $(\epsilon,k_0)=(0.53,3)$, and this fact follows from a striking result of Bauer and Bennett \cite[Corollary 1.6;\,$y=13$]{BaBe} based on the hypergeometric method.
Another key idea for proving Proposition \ref{prop-c7etc} is to find a non-trivial lower bound for the ratio of $z$ to $Z$ in equation \eqref{bcY0zZ}.
Note that exceptional cases in Theorem \ref{c-7etc} for $c=13$ can be effectively determined by Proposition \ref{prop-c7etc} together with the proof of Theorem \ref{c-7etc}, and those cases will be handled completely in the next section.

Below, in order to prove Proposition \ref{prop-c7etc}, we shall prepare a few lemmas, where, to follow the arguments in previous sections, in equation \eqref{bcY0zZ} we define a positive integer $a$ by
\[
a+b=c^z, \ \ a+b^{Y_0}=c^Z.
\]
It suffices to consider when $a>1$ by \cite[Theorem 6]{LeV}.
The above two equations are in this section also referred to the 1st and 2nd equation, respectively.
Note that $\Delta=Y_0-1$.
Further, similarly to the proof of Theorem \ref{c-7etc}, one may assume $e_c(b)=E$, where $E=3$.

\begin{lem}\label{sharpboundforZ/z}
Let $(z,Z)$ be a solution to equation $\eqref{bcY0zZ}.$
If $Z \ge 2k_0+1,$ then $Z<\frac{2}{1-\epsilon} \,z.$
\end{lem}

\begin{proof}
We distinguish two cases according to whether $Z$ is even or odd.

First, suppose that $Z$ is even.
From 2nd equation, it follows that $a$ is divisible by the positive integer $ c^{Z/2}+b^{Y_0/2}$.
In particular, $a \ge c^{Z/2}+b^{Y_0/2}>c^{Z/2}$.
Since $a<c^z$, one has $Z<2z$.

Next, suppose that $Z$ is odd.
From 2nd equation,
\[
a=c^Z-b^{Y_0}
=\bigr(c^{\frac{Z-1}{2}}\sqrt{c}+b^{\frac{Y_0}{2}}\big)\big(c^{\frac{Z-1}{2}}\sqrt{c}-b^{\frac{Y_0}{2}}\bigr).
\]
Dividing both leftmost and rightmost sides above by $(c^{\frac{Z}{2}}+b^{\frac{Y_0}{2}})\,c^{\frac{Z-1}{2}}$ yields
\[
\frac{a}{(c^{\frac{Z}{2}}+b^{\frac{Y_0}{2}})\,c^{\frac{Z-1}{2}} }=\sqrt{c}-\frac{b^{\frac{Y_0}{2}}}{c^{\frac{Z-1}{2}}}.
\]
Since $a<c^z$, it follows that
\[ \ \
0<\sqrt{c}-\frac{b^{\frac{Y_0}{2}}}{c^{\frac{Z-1}{2}}}
<\frac{c^z}{c^{\frac{Z}{2}} \cdot c^{\frac{Z-1}{2}}},
\]
which leads to a restricted rational approximation to $\sqrt{c}$, as follows:
\[
\left|\sqrt{c}-\frac{b^{\frac{Y_0}{2}}}{c^{\frac{Z-1}{2}}}\right|<\frac{1}{c^{\,Z-\frac{1}{2}-z}}.
\]
On the other hand, if $\frac{Z-1}{2} \ge k_0$, then property (iii) of $c$ for $(P,Q,k):=(b^{\frac{Y_0}{2}},c^{\frac{Z-1}{2}},\frac{Z-1}{2})$ is used to give an opposite restriction to the above approximation, as follows:
\[
\left|\sqrt{c}-\frac{b^{\frac{Y_0}{2}}}{c^{\frac{Z-1}{2}}}\right| > \frac{1}{c^{\frac{Z-1}{2} \cdot (1+\epsilon)}}.
\]
These inequalities together lead to
\[
Z-\frac{1}{2}-z<\frac{Z-1}{2} \cdot (1+\epsilon),
\]
so that
\[
\frac{1-\epsilon}{2} \cdot Z<z-\frac{\epsilon}{2}.
\]
Since $\epsilon<1$, this implies the assertion.
\end{proof}

\begin{lem}\label{Y0>4}
If $Y_0>4,$ then equation \eqref{bcY0zZ} has no solution, whenever $b$ exceeds some constant which depends only on $c,\epsilon$ and $k_0$ and is effectively computable.
\end{lem}

\begin{proof}
Let $(z,Z)$ be a solution to equation \eqref{bcY0zZ}.
Similarly to the proof of Theorem \ref{c-7etc}, one finds that $b^2+\delta b+1$ is divisible by $c^z/\gcd(\Delta/E,c^z)$ for some $\delta=\pm1$.
In particular,
\[
b^2+\delta b+1 \ge \frac{c^z}{\gcd(\Delta/E,c^z)}.
\]
Since $\Delta=Y_0-1$, and $Y_0 \ll_c 1$ by assumption, it follows that
\[
2b^2>b^2+\delta b+1 \ge \frac{c^z}{\Delta/E} \gg_{c} c^z,
\]
so that $b \gg_c c^{z/2}$.
This together with the trivial inequality $b<c^{Z/Y_0}$ leads to
\[
c^{z/2} \ll_c c^{Z/Y_0}.
\]
Since $b<c^Z \ll_{c,k_0} 1$ if $Z<2k_0+1$, one may assume by Lemma \ref{sharpboundforZ/z} that $Z<\frac{2}{1-\epsilon}\,z$.
Thus the above displayed inequality implies that
\begin{equation}\label{ineq-ep0.6}
c^{\,\mathcal E z} \ll_c 1,
\end{equation}
where
\[
\mathcal E:=\frac{1}{2} - \frac{2}{(1-\epsilon)Y_0}.
\]
From now on, suppose that $Y_0>E+1\,(=4)$.
Since $Y_0$ is even and $Y_0=\Delta+1$ with $\Delta \equiv 0 \pmod{E}$, it follows that $Y_0 \ge 3E+1=10$, whereby
\[
\mathcal E \ge \frac{1}{2} - \frac{1}{5(1-\epsilon)}>0,
\]
where the last inequality follows from assumption $\epsilon<0.6$.
To sum up, putting $\epsilon':=\frac{1}{2} - \frac{1}{5(1-\epsilon)}$, one has $c^{\,\epsilon'z} \ll_c 1$ by \eqref{ineq-ep0.6}, so that $b<c^z \ll_{c,\epsilon} 1$.
This completes the proof.
\end{proof}

\begin{lem}\label{Zle3z}
Let $(z,Z)$ be a solution to equation \eqref{bcY0zZ} for $Y_0=4.$
Then $Z \le 3z.$
\end{lem}

\begin{proof}
We shall preserve the letter $Y_0$.
Observe from 1st, 2nd equations that $c^{Y_0 z}=(a+b)^{Y_0}>a+b^{Y_0}=c^Z$, so that $Y_0 z>Z$.
Further, one reduces 1st, 2nd equations modulo $a$ to see that
\[
c^{Y_0 z-Z} \equiv 1 \mod{a}.
\]
In particular, $c^{Y_0 z-Z}>a$.
We distinguish two cases.

Suppose that $a \le c^{z-1}$.
Since $c^Z>b^{Y_0}$ and $b=c^z-a \ge (1-1/c)\,c^z$, it follows that
\[
c^Z>(\,(1-1/c)\,c^z\,)^{Y_0}=\frac{1}{\bigr(\frac{c}{c-1}\bigr)^{Y_0}}\,c^{Y_0 z},
\]
which implies that
\[
Y_0>\frac{\log c}{\log(1+\frac{1}{c-1})} \cdot (Y_0 z-Z).
\]
Since $c \ge 4$ and $Y_0 z-Z \ge 1$, the above gives $Y_0>4.81\ldots$, a contradiction.

If $a>c^{z-1}$, then $c^{Y_0 z-Z}>a>c^{z-1}$, so that $Y_0 z-Z \ge z$, i.e., $Z \le (Y_0-1)z=3z$.
\end{proof}

\begin{lem}\label{Zge3z}
Let $(z,Z)$ be a solution to equation \eqref{bcY0zZ} for $Y_0=4.$
Then $Z \ge 3z.$
\end{lem}

\begin{proof}
We have $b(b-1)(b^2+b+1)=c^z(c^{Z-z}-1)$.
Since $b-1$ is not divisible by $c$ (as $e_c(b)=E>1$) and $c$ is a prime, it follows that $b^2+b+1 \equiv 0 \pmod{c^z}$, so that
\begin{gather}
b^2+b+1=Kc^z,\label{K-1}\\
(b^2-b)K=c^{Z-z}-1\label{K-2}
\end{gather}
for some positive odd integer $K$.
The key idea is to derive a non-trivial lower bound for $K$.
For this we follow an idea used in the proofs of \cite[Lemma 9.2]{MiyPin} and \cite[Lemma 5.6]{MiyPin2}.
Note that $Z \ge 2z$ as $b^2>c^z/2$ by \eqref{K-1}, and $c^Z>b^{Y_0}=b^4$ by 2nd equation with $c>4$.

Since, by \eqref{K-1},
\[
b^2-b \equiv -2b-1 \mod{c^z},
\]
one finds from \eqref{K-2} that
\[
(2b+1)K \equiv 1 \mod{c^z}.
\]
Squaring both sides above shows
\[
(2b+1)^2 K^2 \equiv 1 \mod{c^z}.
\]
Since
\[
(2b+1)^2 \equiv 4b^2+4b+1 \equiv -3 \mod{c^z},
\]
it follows that
\[
3 K^2 \equiv -1 \mod{c^z}.
\]
This implies the following lower bound for $K$:
\[
K \ge \sqrt{\frac{4c^z-1}{3}}=\sqrt{\frac{4-1/c^z}{3}} \cdot c^{z/2}.
\]
Then
\[
2 b^2>b^2+b+1=K c^z \ge \sqrt{\frac{4-1/c^z}{3}} \cdot c^{3z/2},
\]
so that
\[
b^4>\frac{1-1/(4c^z)}{3} \cdot c^{3 z}.
\]
Since $c^Z>b^4$ with $c \ge 4$, it follows that
\[
c^{Z-3z}>\frac{1-1/(4c^z)}{3}>c^{-1}.
\]
This leads to the assertion.
\end{proof}

\begin{lem}\label{Y0=4}
If $Y_0=4,$ then equation \eqref{bcY0zZ} has no solution.
\end{lem}

\begin{proof}
For $Y_0=4$, suppose that equation \eqref{bcY0zZ} has a solution $(z,Z)$.
Lemmas \ref{Zle3z} and \ref{Zge3z} together say that $Z=3z$.
Reducing 1st, 2nd equations modulo $a$ implies that $b \equiv 1 \pmod{a}$,
in particular, $b>a$.
Thus,
\[
c^z=a+b<2b<2 c^{Z/Y_0}=2 c^{3z/4},
\]
so that $c^z<16$.
Now brute force computation suffices to verify that equation \eqref{bcY0zZ}: $c^{3z}-c^z=b^4-b$ does not hold in any possible case.
\end{proof}

Proposition \ref{prop-c7etc} follows from the combination of Lemmas \ref{Y0>4} and \ref{Y0=4}.

\section{Case where $c=13$ with $\Delta$ odd} \label{sec-c13-1}

In this and the next section, we put $c=13$.
We may assume that $e_{c}(a)=e_{c}(b)>1$, and write $E=e_{c}(a)=e_{c}(b)$.
Observe that 
\[
e_{c}(h)=3 \ \Leftrightarrow \ h \equiv 3,4,9,10 \pmod{c}. 
\]
It suffices to consider when $\max\{a,b\}>c$ (because Conjecture \ref{atmost1} is proved to be true for each pair $(a,b)$ with $\max\{a,b\}<c$ (cf.~\cite{Sty})).

Let $x,y,z,X,Y,Z$ be positive integers such that
\begin{eqnarray} \label{c13-main}
&a^x+b^y=c^z, \label{1steq-c13}\\
&a^X+b^Y=c^Z. \label{2ndeq-c13}
\end{eqnarray}
As in previous sections, we put $\Delta,\Delta',\mathcal D$ as follows: 
\[
\Delta=|xY-yX|, \ \ \Delta'=\gcd(\Delta/E,c^{\,\min\{z,Z\}}), \ \ \mathcal D=c^{\,\min\{z,Z\}}/\Delta'.
\]
Since $c$ is a prime, there exists some nonnegative integer $n'$ with $n' \le \min\{z,Z\}$ such that
\[
\Delta'=c^{n'}, \ \ \mathcal D=c^{\,\min\{z,Z\}-n'}.
\]
Further, for each $h \in \{a,b\}$, we have
\begin{equation} \label{c13-divrel}
h^{E-1}+\delta_h h^{E-2}+{\delta_h}^2\,h^{E-3}+\cdots +{\delta_h}^{E-1} \equiv 0 \mod{\mathcal D},
\end{equation}
where $\delta_h \in \{1,-1\}$ is defined by the congruence $h^E \equiv \delta_h \pmod{c}$.
This in particular implies that 
\[
m > (\mathcal D /2)^{\frac{1}{E-1}},
\]
where $m:=\min\{a,b\}$.

In the remainder of this section, we assume that $z \le Z$.
Equations \eqref{1steq-c13} and \eqref{2ndeq-c13} are often referred to the 1st and 2nd equation, respectively.

\subsection{Finding explicit bounds for solutions}

Here we shall attempt to find sharp upper bounds for the solutions such that $\Delta$ is odd.
The aim is to show the following lemma.

\begin{lem} \label{c13-Delta-odd}
Suppose that $\Delta$ is odd.
Then the following hold.
\begin{itemize}
\item[(i)] 
$Z<K_1\,\log a\,\log b,$ where
\[ K_1=\begin{cases}
\,2367 & \text{if $m<c,$} \\
\,843 & \text{if $m>c$}.
\end{cases} \]
\item[(ii)] $\Delta<K_1(\log^2 c)\,z.$
\item[(iii)] If $z \ge 9,$ then $z\,Z \le K_2\,\log a\,\log b,$ where $K_2=10459.$
\item[(iv)] $\Delta \le K_3,$ where
\[ K_3=\begin{cases}
\,77862 & \text{if $m<c$ and $z \le 8,$} \\
\,44368 & \text{if $m>c$ and $z \le 8,$} \\
\,68809 & \text{if $z \ge 9.$}
\end{cases} \]
Further, $n' \le 3.$
\item[(v)] If $\max\{x,y\} \ge 3,$ then $z \le 9.$
\item[(vi)] If $\max\{x,y\}=2,$ then 
\[ z \le \begin{cases}
\,23650 & \text{if $n'=0$}, \\
\,23651 & \text{if $n'=1$}, \\
\,47322 & \text{if $n'=2$}, \\
\,70986 & \text{if $n'=3$}.
\end{cases} \]
\item[(vii)] $x>1$ or $y>1.$
\end{itemize}
\end{lem}

Before going to the proof of the above lemma, we make some remarks on its last two statements.
For (vi), while we follow the proof of Lemma \ref{c-prime-xleE-2--yleE-2} for $E=3$, which relies on explicit version of the size of greatest prime factor of the values of polynomials with integer coefficients at integer points, we adopt another way using the non-Archimedean analogue to Baker's method over the ring of Eisenstein integers in a non-artificial sense, which is regarded as an analogue to a bright idea of Luca \cite{Lu} over the ring of Gaussian integers.
For (vii), in order to show that $\min\{X,Y\}=1$, we rely on modular approach recently established by Bennett and Siksek \cite{BeSi} instead of a strategy used in the proof of Lemma \ref{c-general-minxy=1--minXY=1}. 

\begin{proof}[Proof of Lemma $\ref{c13-Delta-odd}$] 
Note that $E=3$, and that
\[
a \equiv 3,4,9,10 \pmod{c}, \ \ b \equiv 3,4,9,10 \pmod{c}.
\]
In particular, either $m \in \{3,4,9,10\}$ or $m>c$.
\par
(i) We follow the proof of Lemma \ref{1st-app} on solution $(X,Y,Z)$.
Then
\[
Z \le \nu_{c}(a^{2X}-b^{2Y}) \le \frac{53.6\cdot 2E\cdot \log a'\,\log b'}{\log^4 c} \cdot {\mathcal B'}^2,
\]
where $a'=\max\{a,c\},\,b'=\max\{b,c\}$,
\[
\mathcal{B'}=\log\,\max \bigr\{\,4\exp(0.64)(\log^2 c)\,T,\,c^4 \,\bigr\}
\]
with $T=\frac{Z}{\log a\,\log b}$.
Thus
\[
T<f\cdot \frac{53.6\cdot 2E}{\log^4 c} \cdot \mathcal{B'}^2,
\]
where $f=\frac{\log a'}{\log a}\cdot \frac{\log b'}{\log b}$.
Since $\max\{a,b\}>c$, one has
\[ 
f=\frac{\log\,\max \{m,c\}}{\log m}=
\begin{cases}
\,\frac{\log c}{\log m} & \text{if $m<c,$} \\
\,1 & \text{if $m>c.$}
\end{cases} \]
These together imply that $T<K_1$.
\par
(ii) This follows from inequality \eqref{Delta-basic-ub} together with (i).
\par
(iii) Firstly we follow the proof of Lemma \ref{calD-l}.
Since $\Delta' \le \Delta/E<K_1/E\cdot(\log^2 c)\,z$ by (ii), one has
\[
\frac{\log \mathcal D}{z}
> \log c-\frac{\log\bigr( K_1/E \cdot (\log^2{c}) \, z \bigr)}{z}.
\]
Assume that $z \ge 9$.
Then the above displayed inequality implies
\begin{equation}\label{c13-Deltaodd-logcalD-l}
\log{\mathcal D}>\zeta \,z,
\end{equation}
where
\[ \zeta=\begin{cases}
\,1.37 & \text{if $m<c,$} \\
\,1.48 & \text{if $m>c.$}
\end{cases}\]
Since $\mathcal D>\exp(\zeta\,z) \ge \exp(1.37\cdot9)>2.2 \cdot 10^5$, it follows that $m=\min\{a,b\} > (\mathcal D /2)^{\frac{1}{E-1}}>13\,(=c)$.
Thus,
\begin{gather}
\mathcal D>\exp(\zeta\,z)>6 \cdot 10^5,
\label{c13-calD-l}\\
m > (\mathcal D /2)^{\frac{1}{E-1}}> (\,\exp(\zeta \,z) /2\,)^{\frac{1}{E-1}}
\label{c13-m-l}
\end{gather}
with $\zeta=1.48$.

Next, we closely follow the proof of Lemma \ref{2nd-app}. 
Then
\[
Z/z-1 < \nu_{\mathcal D}(a^{2X}-b^{2Y}) \le \frac{53.6\cdot 2E\cdot \log a'\,\log b'}{\log^4 \mathcal D} \cdot {\mathcal B'}^2,
\]
where $a':=\max\{a,\mathcal D\},\, b':=\max\{b,\mathcal D\}$, 
\[
\mathcal{B'}:=\log\,\max \bigr\{ \,4\,{\rm e}^{0.64}(\log c)(\log \mathcal D)\,T/z, \, {\mathcal D}^4\, \bigr\}
\]
with ${\rm e}=\exp(1)$ and $T:=\frac{z\,Z}{\log a\,\log b}$. 
These together imply
\[
T<\frac{53.611\cdot 2E\cdot f\,z^2}{\log^4 \mathcal D}\cdot \mathcal{B'}^2,
\]
where $f:=\frac{\log a'}{\log a}\cdot \frac{\log b'}{\log b}$.
Observe that
\[
f \le \max\biggr\{1,\,\frac{\log^2 \mathcal D}{\log^2 m}\biggr\}.
\] 

If $4\,{\rm e}^{0.64}(\log c)(\log \mathcal D)\,T/z>\mathcal D^4$, then 
\[
\frac{\mathcal D^4}{4\,{\rm e}^{0.64}(\log c)\log \mathcal D}<T/z<\frac{53.611\cdot 2E\cdot f\,z}{\log^4 \mathcal D}\,\log^2 \bigr(\,4\,{\rm e}^{0.64}(\log c)(\log \mathcal D)\,T/z\,\bigr).
\]
Since $\mathcal D^4>(\,4\,{\rm e}^{0.64}(\log c)\log \mathcal D\,)^9$ by \eqref{c13-calD-l}, it follows from 
the left-hand inequality displayed above, one sees that 
\[
(T/z)^{1/8} \ge 4\,{\rm e}^{0.64}(\log c)\log \mathcal D.
\]
In particular, $T/z>2 \cdot 10^{19}$.
Therefore, by the second one, 
\[
T/z<\frac{53.611\cdot 2E\cdot f\,z}{\log^4 \mathcal D}\, \log^2 \bigr(\,(T/z)^{9/8}\bigr).
\]
Since $z \ge 9$, $m > (\mathcal D /2)^{\frac{1}{E-1}}$ by \eqref{c13-m-l}, and $\log \mathcal D>\zeta\,z$ by \eqref{c13-Deltaodd-logcalD-l}, it follows that
\begin{align*}
\frac{T/z}{\log^2 \bigr(\,(T/z)^{9/8}\bigr)}
&<53.611\cdot 2E\cdot \max\biggr\{\frac{z}{\log^4 \mathcal D},\,
\frac{z}{\log^2 m\,\log^2 \mathcal D}\biggr\}\\
&<53.611\cdot 2E\cdot \max\biggr\{\frac{1}{\zeta^4 z^3},\,
\frac{1}{\zeta^2(\log^2 m)\,z}\biggr\}
\ < \ 0.41,
\end{align*}
which is a contradiction as $T/z$ is large.

Finally, if $4\,{\rm e}^{0.64}(\log c)(\log \mathcal D)\,T/z \le \mathcal D^4$, then
\begin{align*}
T&<\frac{53.611\cdot 2E\cdot f\,z^2 \cdot 4^2}{\log^2 \mathcal D}\\
&<53.611\cdot 2E\cdot \max\biggr\{\frac{z^2}{\log^2 \mathcal D},
\frac{z^2}{\log^2 m}\biggr\} \cdot 4^2\\
&<53.611\cdot 2E\cdot \max\biggr\{\frac{1}{\zeta^2},
\frac{z^2}{\log^2 m}\biggr\} \cdot 4^2 
\ \le \ K_2.
\end{align*}
\par
(iv) Firstly, since $\Delta<K_1(\log^2 c)\,z$ by (ii), if $z \le 8$, then
\begin{equation} \label{c13-iv-1}
\Delta \le
\begin{cases}
\,124579 & \text{if $m<c,$} \\
\,44368 & \text{if $m>c.$}
\end{cases}
\end{equation}
If $z \ge 9$, then, since $zZ \le K_2\,\log a\,\log b$ by (iii), it follows that $\Delta<\frac{z\,Z}{\log a \, \log b} \log^2 c \le K_2\log^2 c< 68810$. 
Thus the inequality $\Delta \le 124579$ always holds. 
Also, $n' \le 4$, as $c^{n'}=\Delta' \le \Delta/E<41527<c^5$.

We shall reduce the bound for $\Delta$ in \eqref{c13-iv-1} when $m<c$. 
By congruence \eqref{c13-divrel} with $h=m$, 
\begin{equation} \label{c13-divrel-m}
c^{z-n'} \mid (m^2+\delta_m \,m+1).
\end{equation}
If $m<c$, then $\nu_c(m^2+\delta_m \,m+1) \le 1$, so that $c^{z-n'} \mid c,$ whereby $z \le 1+n' \le 5$.
Then $\Delta<2367 \cdot (\log^2 c) \cdot 5<77863$. 
To sum up, the asserted bounds for $\Delta$ hold.
Further observe that $n' \le 3$.
\par
(v) We basically follow the proof of Lemma \ref{c-prime-xleE-2--yleE-2} for the case where $\max\{x,y\} \ge E\,(=3)$.
Relation \eqref{c13-divrel-m} implies that $c^{z-n'} \le m^2+m+1$.
Since $n' \le 3$ by (iv), and $m<c^{\,z/\max\{x,y\}} \le c^{z/3}$ by assumption that $\max\{x,y\} \ge 3$, it follows that $c^{z-3}<c^{2z/3}+c^{z/3}+1$, leading to $z \le 9$.
\par
(vi) Firstly, we follow the proof of Lemma \ref{c-prime-xleE-2--yleE-2} for the case where $\max\{x,y\}=E-1\,(=2)$.
By \eqref{c13-divrel-m}, there is a positive integer $K$ such that
\begin{equation} \label{c13-vi-0}
m^2+\delta_m \,m+1=K c^{z-n'}.
\end{equation}
Since $a \ne b$, and $\min\{x,y\}=1$ as $\Delta$ is odd, one finds that
\[
K c^{z-n'}=m^2+\delta_m \,m+1 \le m^2+m+1 \le m^2+\max\{a,b\} \le c^z,
\]
so that
\begin{equation} \label{Kup}
K \le c^{n'}.
\end{equation}

Below, we shall argue over the ring of Eisenstein integers.
Let ${\mathcal O}$ be the ring of integers of $\mathbb{Q}(\omega)$, where $\omega=\exp(2\pi i/3)$.
One factorizes the left-hand side of \eqref{c13-vi-0} to see that
\begin{equation} \label{c13-vi-2}
f \cdot \bar{f}=K' \cdot c^{z_1},
\end{equation}
where 
\[
f=m -\delta_m\,\omega, \ \ \bar{f}=m -\delta_m\,\bar{\omega}, 
\ \ K'=K/c^{\,\nu_c(K)}, \ \ z_1=z-n'+\nu_c(K).
\]
Note that $z_1>0$. 
Observe that
\begin{equation} \label{c13-vi-3}
f - \bar{f} = -\delta_m\,\omega (1-\omega).
\end{equation}
In particular, $f-\bar{f}$ is not divisible by any rational prime. 
From this, it is not hard to see that any rational prime dividing the right-hand side of \eqref{c13-vi-2} has not to be inert, namely, it has to be factorized into the product of two prime elements. 
This implies that
\[
f= K_0 \cdot \pi^{z_1}
\]
for some $K_0 \in {\mathcal O}$ with $K_0\, \overline{K_0}=K'$, and $\pi$ is a prime element with $\pi \bar{\pi}=c$.
Note that $\pi \ne \bar{\pi}$ (as $c \ne 3$). 
Equality \eqref{c13-vi-3} becomes
\[
K_0 \cdot \pi^{z_1}-\overline{K_0} \cdot {\bar{\pi}}^{\,z_1}= -\delta_m\,\omega (1-\omega).
\]
Reducing this modulo $\pi^{z_1}$ gives
\[
\overline{K_0} \cdot {\bar{\pi}}^{\,z_1} -\delta_m\,\omega (1-\omega) \equiv 0 \mod{\pi^{z_1}}.
\]
Putting 
\[
\alpha_1:=\bar{\pi}, \ \ \alpha_2:=\delta_m\,\omega (1-\omega)/\,\overline{K_0},
\]
one has
\[
\nu_{\pi} ({\alpha_1}^{z_1} - {\alpha_2}) \ge z_1.
\]
Now we shall apply Proposition \ref{BL} to find an upper bound for the left-hand side above. 
For this, set $(b_1,b_2):=(z_1,1)$.
Observe that $p=c,\,D=2,\,f_{\pi}=1$.
It is easy to see that $\alpha_1,\alpha_2$ are multiplicatively independent from the fact that $\pi$ divides neither $K_0$ nor $1-\omega$.
One may choose ${\rm g}:=\pi \cdot \bar{\pi}-1=c-1$. 
Since ${\rm h}(\alpha_1)=\frac{1}{2}\log c$, one may take $H_1:=\log c$.
Further, observe that ${\rm h}(\alpha_2)={\rm h}({\alpha_2}')$, where ${\alpha_2}'=(1-\omega)/\,\overline{K_0}$, and that the minimal polynomial $g_2(t)$ of ${\alpha_2}'$ over $\mathbb Z$ is given by
\[
g_2(t)=
\begin{cases}
K' t^2 - 3u t +3 & \text{if $3 \nmid K'$},\\
(K'/3)\,t^2 - u t +1 & \text{if $3 \mid K'$},
\end{cases}
\]
where $u$ is the rational integer such that ${\alpha_2}'=u+v\,\omega$ with some $v \in \mathbb Z$.
Since $|{\alpha_2}'|=\sqrt{3/K'}$, it turns out that
\begin{align*}
{\rm h}( {\alpha_2}')
&\le
\frac{1}{2}\,\Bigr(\!\log K' + 2\,\log\,\max \{\,1,|{\alpha_2}'|\,\}\Bigr)=\frac{1}{2}\log\, \max \{K',3\}.
\end{align*}
Since $K' \not\equiv 0 \pmod{c}$ and $K'$ is a divisor of the left-hand side of \eqref{c13-vi-0}, it follows from inequality \eqref{Kup} that
\[
K' \le K_u,
\]
where $K_u =c^{n'}$.
Therefore, 
\[
\frac{D}{f_\pi}\,{\rm h}(\alpha_2)=2\,{\rm h}({\alpha_2}') \le \log K_u, 
\]
thereby one may take $H_2:=\log\,\max\{K_u,c\}$.
Then
\[
\nu_{\pi} ({\alpha_1}^{z_1} - {\alpha_2}) < \frac{27.3 \cdot 2^2\,c\,H_2}{\log^3 c} \cdot \mathcal B^2
\]
with $\mathcal B=\max\{\, \log (z_1/H_2+1/\log c)+\log \log c+0.4,\,4\log c \,\}$.
Finally, the two bounds for $\nu_{\pi} ({\alpha_1}^{z_1} - {\alpha_2})$ together lead to the bounds for $z_1\,(=z-n')$, implying the assertion. 
\par
(vii) Suppose that $x=y=1$.
Since $\{a,b\} \ne \{3,10\}$ as $\max\{a,b\}>c=13$, by Proposition \ref{twoclass}, there is no loss of generality in assuming that $X$ is odd and $Y$ is even.
Equation \eqref{2ndeq-c13} is rewritten as
\[
(b^{Y/2})^2-13^Z=(-a)^X.
\]
We can directly apply the following result of Bennett and Siksek which is a consequence of \cite[Theorems 2 and 5;\,$q=13$]{BeSi}.

\begin{prop}\label{BeSi-LeNa}
All solutions to the equation
\[
S^2 - 13^k = T^n
\]
in integers $S,T,k$ and $n$ with $S>0, \gcd(S,13)=1, k \ge 1, n \ge 3$ are given by $(S,T,k,n)=(16,3,1,5),(14,3,2,3)$ and $(499,12,2,5).$
\end{prop}
It turns out that $X=1$, so that
\begin{equation} \label{c13-vii-2}
c^Z-c^z=b^Y-b
\end{equation}
with $c^z>b$.
Since $Y-1=\Delta \le K_3$ by (iv), and $\Delta \equiv 0 \pmod{E}$ with $\Delta$ odd, it holds that
\begin{equation} \label{c13-vii-0}
Y \le K_3+1, \ \ Y \equiv 4 \pmod 6.
\end{equation}
Thus equation \eqref{c13-vii-2} is the same as \eqref{bcY0zZ}, and we can use the arguments in Section \ref{sec-eff} for $c=13$ with $(\epsilon,k_0)=(0.53,3)$.

First, Lemma \ref{Y0=4} says that $Y>4$. 
Next, we shall apply Lemma \ref{sharpboundforZ/z}.
If $Z \ge 7$, then $Z \ge 2k_0+1\,(=7)$, so that
\[
Z < \frac{200}{47}\,z.
\]
On the other hand, from the proof of Lemma \ref{Y0>4}, 
\[
c^{z-n'}=\frac{c^z}{\Delta'} <2 b^2<2 c^{2Z/Y}.
\]
These together imply
\[
c^{\,(1-\frac{400}{47Y})\,z-n'}<2.
\]
Since $n' \le 3$ by (iv), and $Y \ge 10$ by \eqref{c13-vii-0}, one concludes from the above displayed inequality that $z \le 21$. 
Thus $Z<\frac{200}{47}\,z<90$.
To sum up, one always has $Z<90$.
Now brute force computation suffices to verify that equation \eqref{c13-vii-2} does not hold in any possible case.
\end{proof}

\subsection{Sieving the remaining cases}
Here we sieve the remaining cases in Lemma \ref{c13-Delta-odd}\,(v,\,vi). 
Based on restrictions provided by Lemma \ref{c13-Delta-odd} and its proof, we check whether the system of 1st and 2nd equations holds, according to assumptions $z \le 9$ in (v) and $z \le 70986$ in (vi), respectively.
For this purpose, it suffices to consider only the case where $x \ge y$ by symmetry of $a$ and $b$. 

We know from congruence \eqref{c13-divrel} with $h=a$ there is a positive integer $t=t(a)$ such that $a^2+\delta_a \, a+1=tc^{z-n'}$, whence
\begin{equation}\label{check3}
(2h+\delta_a)^2=4tc^{z-n'}-3.
\end{equation}
This implies that
\begin{equation}\label{check4}
t \not\equiv 0 \pmod{2}, \ \ t \not\equiv 0 \pmod{9}, \ \ p \not\equiv 2 \pmod{3}
\end{equation}
for any odd prime factor $p$ of $t$.
Further, since $a<c^{z/x} \le c^{z/2}$ as $x=\max\{x,y\} \ge 2$, and $t=(a^2+\delta_a \,a+1)/c^{z-n'} \le (a^2+a+1)/c^{z-n'}$, observe that
\begin{gather}\label{check5}
t< (1+1/c^{z/2}+1/c^z)\,c^{n'}.
\end{gather}

\subsubsection{Case {\rm (v)}}\label{sec-case v}
We perform the algorithm consisting of the following three steps.
The computation time was a few seconds.

\vspace{0.2cm}\noindent 
{\bf Step 1.} {\it Find all possible numbers $z,n',t.$}

\vspace{0.1cm}
We generate a list including all possible values of $z,n'$ and $t$ satisfying \eqref{check3}, \eqref{check4} and \eqref{check5}.
The fact that $4tc^{z-n'}-3$ is a square following from \eqref{check3} is a key sieving relation.
It turns out that the resulting list, say $list1$, contains 114 elements.

\vspace{0.2cm}{\tt
begin
\vskip.1cm 
\hskip.2cm 
$z_u:=9$
\vskip.1cm
\hskip.2cm 
for $n':=0$ to $3$ do
\vskip.1cm
\hskip.2cm 
for $z:=\max\{1,n'\}$ to $z_u$ do
\vskip.1cm 
\hskip.2cm 
if $z \ge 7$ then $t_u:=c^{n'}$; else $t_u:= \lfloor (1+1/c^{z/2}+1/c^z)\,c^{n'}\rfloor$
\vskip.1cm
\hskip.2cm 
for $t:=1$ to $t_u$ do
\vskip.1cm
\hskip.2cm 
Sieve with \eqref{check4}
\vskip.1cm
\hskip.2cm 
if $4tc^{z-n'}-3$ is a square then
\vskip.1cm
\hskip.2cm 
Put $[z,n',t]$ into $list1$
\vskip.1cm
end}

\vspace{0.2cm}\noindent 
{\bf Step 2.} {\it Find all possible numbers $a,b,x,y,z,n'.$}

\vspace{0.1cm}
We generate a list including all possible values of $a,b,x,y,z$ and $n'$ by using 1st equation.
It turns out that the resulting list, say $list2$, contains 108 elements.

\vspace{0.4cm}{\tt
begin
\vskip.1cm
\hskip.2cm $x_l:=3$ 
\vskip.1cm
\hskip.2cm for each element $[z,n',t]$ in $list1$ do
\vskip.1cm
\hskip.2cm $\mathcal D:=c^{z-n'}$; $A:=\sqrt{4t\mathcal D-3}$
\vskip.1cm
\hskip.2cm $m_l:=\max\bigr\{ \,3, \bigr\lceil \,(\,\mathcal D/2\,)^{\frac{1}{E-1}} \,\bigr\rceil \, \bigr\}$
\vskip.1cm
\hskip.2cm for $\delta_a$ in $[-1,1]$ do
\vskip.1cm
\hskip.2cm $a:=(A-\delta_a)/2$
\vskip.1cm
\hskip.2cm if $a \ge m_l$ and $a \le \lfloor c^{z/x_l} \rfloor$ and $a \equiv 3,4,9,10 \pmod{c}$
\vskip.1cm
\hskip.2cm and $a^2+\delta_a \, a+1 \equiv 0 \pmod{c^{z-n'}}$ then
\vskip.1cm
\hskip.2cm for $x:=x_l$ to $\bigr\lfloor \frac{\log c}{\log a}\,z \bigr\rfloor$ do
\vskip.1cm
\hskip.2cm for $y:=1$ to $x$ do
\vskip.1cm
\hskip.2cm if $c^z-a^x$ is a $y$\,th power then
\vskip.1cm
\hskip.2cm $b:=(c^z-a^x)^{1/y}$
\vskip.1cm
\hskip.2cm if $b \ge m_l$ and $b \equiv 3,4,9,10 \pmod{c}$
\vskip.1cm \hskip.2cm 
and $b^3 \equiv \pm1 \pmod{c^{z-n'}}$ 
\vskip.1cm \hskip.2cm 
and $b \not\equiv a \pmod{2}$ and $\gcd(b,a)=1$ then
\vskip.1cm \hskip.2cm 
Put $[a,b,x,y,z,n']$ into $list2$
\vskip.1cm
end}

\vspace{0.2cm}\noindent 
{\bf Step 3.} {\it Find all possible numbers $a,b,x,y,z,X,Y,Z.$}

\vspace{0.1cm}
We search all possible values of $a,b,x,y,z,X,Y$ and $Z$ by using 2nd equation.
For this, we use Lemma \ref{c13-Delta-odd}\,(i,\,ii,\,iv) and the latter two inequalities in \eqref{trivial-ineqs}, and distinguish two cases according to whether $\Delta=xY-Xy$ or $Xy-xY$.
It turns out that the output is the empty list.

\vspace{0.2cm}{\tt
begin
\vskip.2cm \hskip.2cm 
for each element $[a,b,x,y,z,n']$ in $list2$ do
\vskip.1cm \hskip.2cm
$X_u:=\lfloor K_1 \log b\,\log c \rfloor$;
$Y_u:=\lfloor K_1 \log a\,\log c \rfloor$;
\vskip.1cm \hskip.2cm
$\Delta_u:=\min\bigr\{\,\lfloor K_1 (\log^2 c)\,z \rfloor,\, K_3 \,\bigr\}$
\vskip.2cm \hskip.4cm 
(Case where $\Delta=xY-Xy$)
\vskip.1cm \hskip.4cm 
for $X:=1$ to $\min\bigr\{X_u, \bigr\lfloor (x Y_u - E c^{n'}) / y \bigr\rfloor\,\bigr\}$ do
\vskip.1cm \hskip.4cm 
for $k:=1$ to $\bigr\lfloor\,\min\{ \Delta_u, \,x \cdot Y_u-y \cdot X\}/(E\, c^{n'})
\ \bigr\rfloor$ by 2 do
\vskip.1cm \hskip.4cm 
$xY:=y \cdot X+k\cdot E \cdot c^{n'}$
\vskip.1cm
\hskip.4cm if $xY \equiv 0 \pmod{x}$ then
\vskip.1cm
\hskip.4cm $Y:=xY$ div $x$
\vskip.1cm
\hskip.4cm if $a^X+b^Y$ is a power of $c$ then
\vskip.1cm
\hskip.4cm $Z:=\log(a^X+b^Y) / \log c$
\vskip.1cm
\hskip.4cm Print $[a,b,x,y,z,X,Y,Z]$
\vskip.3cm \hskip.4cm 
(Case where $\Delta=Xy-xY$)
\vskip.1cm \hskip.4cm 
for $Y:=1$ to $\min\bigr\{Y_u, \bigr\lfloor (X_u y - E c^{n'}) / x \bigr\rfloor\,\bigr\}$ do
\vskip.1cm \hskip.4cm 
for $k:=1$ to $\bigr\lfloor\,\min\{ \Delta_u, \,y \cdot X_u-x \cdot Y\}/(E\,c^{n'}) \ \bigr\rfloor$ by 2 do
\vskip.1cm
\hskip.4cm $yX:=x \cdot Y+k\cdot Ec^{n'}$
\vskip.1cm
\hskip.4cm if $yX \equiv 0 \pmod{y}$ then
\vskip.1cm
\hskip.4cm $X:=yX$ div $y$
\vskip.1cm
\hskip.4cm if $a^X+b^Y$ is a power of $c$ then
\vskip.1cm
\hskip.4cm $Z:=\log(a^X+b^Y) / \log c$
\vskip.1cm
\hskip.4cm Print $[a,b,x,y,z,X,Y,Z]$
\vskip.4cm 

end}

\subsubsection{Case {\rm (vi)}}
This case is dealt with by performing the same algorithm as in Section \ref{sec-case v} by resetting the value of $x$ as $x:=2$ and resetting the value of $z_u$ as $z_u:=23650, 23651, 47322$ or $70986$, according to whether $n'=0,1,2$ or $3$, respectively.
The computation time was less than 1 hour. 
We mention that it turned out from Step 1, which was the most time-consuming part, that none of the numbers $4tc^{z-n'}-3$ is a square for $z \ge 8$.

This finishes studying system of equations \eqref{1steq-c13} and \eqref{2ndeq-c13} for $c=13$ with $\Delta$ odd.

\section{Case where $c=13$ with $\Delta$ even} \label{sec-c13-2}

With the same notation as in the previous section, we consider the system of equations \eqref{1steq-c13} and \eqref{2ndeq-c13}.
Suppose that $\Delta$ is even.
We shall observe that this leads to a contradiction.

From the proof of Lemma \ref{c-prime-parity}, either $x$ and $y$ are even, or $X$ and $Y$ are even.
Since the following argument hold in either case, we will assume that both $X$ and $Y$ are even.
Write $X=2X'$ and $Y=2Y'$.
Equation \eqref{2ndeq-c13} becomes
\begin{eqnarray} \label{2ndeq'-c13-dleta-even}
a^{2X'}+b^{2Y'}=c^Z.
\end{eqnarray}
Below, we follow the first part of the proof of \cite[Theorem 3]{MiyPin2}.

We can write 
\[
c=\beta \cdot \bar{\beta},
\]
where $\beta=2+3i$.
By a usual factorization argument to equation \eqref{2ndeq'-c13-dleta-even} over $\mathbb Z[i]$ (see the proof of \cite[Lemma 7.4\,(i)]{MiyPin2}), one can show that
\begin{equation}\label{gauss-fac}
\{a^{X'},b^{Y'}\}=\{a(Z),b(Z)\},
\end{equation}
where 
\[
a(Z)=\dfrac{1}{2}\,\bigr|\beta^Z+(-\bar{\beta})^Z\bigr|, \quad
b(Z)=\dfrac{1}{2}\,\bigr|\beta^Z-(-\bar{\beta})^Z\bigr|.
\]

\begin{lem}\label{ecvalues}
The following hold.
\[
\bigr(\,e_{c}(\,a(Z)\,),\,e_{c}(\,b(Z)\,)\,\bigr)=\,
\begin{cases}
\,(3,6) & \text{if $Z \equiv 1,3 \pmod{6}$},\\
\,(2,1) & \text{if $Z \equiv 2 \pmod{6}$},\\
\,(6,3) & \text{if $Z \equiv 0,4 \pmod{6}$},\\
\,(1,2) & \text{if $Z \equiv 5 \pmod{6}$}.
\end{cases}
\]
\end{lem}

\begin{proof}
We show the assertion only for the value of $e_{c}(\,a(Z)\,)$.
It is easy to verify the assertion for $Z=1,2,...,6$.
Thus it suffices to show that 
\begin{equation} \label{claim}
a(Z+6) \equiv \pm a(Z) \mod{c} \ \ \ (Z=1,2,\ldots).
\end{equation}
For this, one can use some congruence properties of the Lucas sequence of the first and second kinds corresponding to the pair $(\beta,\bar{\beta})$.
Define two integer sequences $\{U_{n}\}_{n=1}^{\infty}$ and $\{V_{n}\}_{n=1}^{\infty}$ as follows:
\[
U_{n}
:= \frac{\beta^n-{\bar{\beta}}^n}{\beta-\bar{\beta}}, \ \ 
V_{n}
:= \beta^n+{\bar{\beta}}^n.
\]
With these notation, since $\beta-\bar{\beta}=6i$, one may write
\[
2a(Z) = \begin{cases}
\,\bigr|\beta^Z-{\bar{\beta}}^Z\bigr|=6\,|\,U_{Z}\,| &\text{if $Z$ is odd,}\\
\,\bigr|\beta^Z+{\bar{\beta}}^Z\bigr|=|\,V_{Z}\,| &\text{if $Z$ is even.}
\end{cases}
\]
It is known that
\[
U_{n} \equiv V_{n-1} \mod{c}, \quad V_{n} \equiv P^n \mod{c} \quad (n=1,2,\ldots),
\] 
where $P=\beta+\bar{\beta}=4$ (cf.~\cite[(IV.12)]{Ri-book2}).
These together imply that
\[
\begin{cases}
a(Z) \equiv \pm 4^{Z+1} \!\!\! \pmod{c} &\text{if $Z$ is odd,}\\
a(Z) \equiv \pm 7 \cdot 4^Z \!\!\! \pmod{c} &\text{if $Z$ is even.}
\end{cases}
\]
This together with the fact that $e_{c}(4)=3$ implies congruence \eqref{claim}.
\end{proof}

Since $e_{c}(a)=e_{c}(b)$, and 
\[
e_{c}(a^{X'})=\frac{e_{c}(a)}{\gcd(\,e_{c}(a),X'\,)}, \ \ e_{c}(b^{Y'})=\frac{e_{c}(b)}{\gcd(\,e_{c}(b),Y'\,)}
\]
(cf.~\cite[Lemma 2.1\,(iii)]{MiyPin2}), it follows from Lemma \ref{ecvalues} that at least one of $X'$ or $Y'$ is even.
Then a simple application of the works \cite{BeElNg,Br,El} on the generalized Fermat equation of signature $(2,4,r)$ with $r \ge 4$ to equation \eqref{2ndeq'-c13-dleta-even} yields $Z \le 3$.
We shall apply relation \eqref{gauss-fac} for $Z=1,2,3$.
It turns out that if $Z \ne 3$, then both $a^{X'}$ and $b^{Y'}$ are not perfect powers, in particular $X'=Y'=1$.
Further, $a \in \{3,9\}$ and $b=46$ for $Z=3$.
In any case, one observes that $e_{c}(a) \ne e_{c}(b)$, contradicting the premise.

The contents of this and the previous sections together complete the proof of Theorem \ref{c13}.

\section{Application of abc-conjecture} \label{sec-abc}

Here we give some applications of $abc$-conjecture to Conjecture \ref{atmost1}, which strengthen many of the lemmas in Sections \ref{sec-c-general} and \ref{sec-c-prime}.
We begin by quoting the following result of Bugeaud and Luca \cite{BuLu}. 

\begin{prop}[Theorem 6.1 of \cite{BuLu}]\label{BuLu-abc}
Under the truth of abc-conjecture, there are only finitely many solutions to the equation
\[
\mathcal A^{x_1} - \mathcal A^{x_2} = \mathcal B^{y_1} - \mathcal B^{y_2}
\]
in positive integers $\mathcal A,\mathcal B,x_1,x_2, y_1$ and $y_2$ such that $\mathcal A>1, \mathcal B>1, x_1 \ne x_2$ and $\mathcal A^{x_1} \ne \mathcal B^{y_1}.$
\end{prop}

It should be noted that the conclusion of the above proposition ensures that the number of exceptional triples $(a,b,c)$ for Conjecture \ref{atmost1pillai} is finite.

Supposing that equation \eqref{abc} has two solutions $(x,y,z), (X,Y,Z)$ with $z \le Z$, the following two lemmas are proved by using Proposition \ref{BuLu-abc}.
We emphasize that any Vinogradov notation appearing in their proofs is not always effective on its implied constant.

\begin{lem}\label{abc-xneX-yneY}
Assume that $abc$-conjecture is true.
Further, assume that $\max\{a,b,c\}$ exceeds some absolute constant $\mathcal C'.$
Then the following hold.
\begin{itemize}
\item[\rm (i)]
$x \ne X,\,y \ne Y,\,z \ne Z.$
\item[\rm (ii)]
$\min\{x,y,X,Y\}=1.$
\end{itemize}
\end{lem}

\begin{proof}
(i) The proof proceeds similarly to that of Lemma \ref{c-prime-xneXandyneY}.
If $z=Z$, then
\[
a^X - a^x = b^y - b^Y.
\]
Proposition \ref{BuLu-abc} for $(\mathcal A,\mathcal B):=(a,b),\, (x_1,x_2, y_1, y_2):=(X,x,y,Y)$ is applied to see in particular that $a,b,x,y \ll 1$, so that $c \le c^z=a^x+b^y \ll 1$.
Similarly, the same conclusion holds when $x=X$ or $y=Y$.
\par
(ii) Let $\varepsilon>0$ be any constant. 
A direct application of $abc$-conjecture to 2nd equation implies
\[
c^Z \ll_{\varepsilon} (abc)^{1+\varepsilon}.
\]
Since $a<c^p$ and $b<c^q$, where $p=\min\{z/x,Z/X\}$ and $q=\min\{z/y,Z/Y\}$, it follows that $c^Z \ll_{\varepsilon} (c^{\,p+q+1})^{1+\varepsilon}$, so that
\begin{equation} \label{abc-pq}
c^{\,\mathcal E} \ll_{\varepsilon} 1,
\end{equation}
where
\[
\mathcal E:=Z-(p+q)(1+\varepsilon) \,-1-\varepsilon.
\]
In particular, $\mathcal E \ll_{\varepsilon} 1$. 

Suppose that $x>1,\,y>1$. 
Since $p \le z/x$ and $q \le z/y$, and one may assume by (i) that $Z \ge z+1$, it follows that
\[
\mathcal E \ge \bigg(1-\biggr(\frac{1}{x}+\frac{1}{y}\bigg)(1+\varepsilon) \bigg)\,z \,-\varepsilon.
\]
Thus, if $x>2$ or $y>2$, then, since $1/x+1/y \le 1/2+1/3=5/6$, one takes $\varepsilon=1/77$ in \eqref{abc-pq} to find that
\[
\mathcal E \ge \biggr(\frac{1}{6}-\frac{5\varepsilon}{6} \bigg)\,z\,-\varepsilon
= \frac{72 z-6}{462} \ge \frac{z}{7}.
\]
Since $\mathcal E \gg z$, it follows from inequality \eqref{abc-pq} that $c^z \ll 1$, so that $a,b,c \ll 1$.
Thus, we may assume that $x=y=2$.
Also, $X \ne Y$ as $\Delta=xY-Xy$ is nonzero.

Suppose that $X>1,\,Y>1$.
Since one may assume by (i) that $X \ne x$ and $Y \ne y$, it follows that $X \ge 3$ and $Y \ge 3$, with $\max\{X,Y\} \ge 4$. 
Since $p \le Z/X$ and $q \le Z/Y$, putting $\mu:=1/X+1/Y$, one has
\[
\mathcal E \ge \big(\,1-(1+\varepsilon)\mu\,\big)\,Z \,-(1+\varepsilon).
\]
Thus, if
\begin{equation} \label{ineq-abc-epmuZ}
1-(1+\varepsilon)\mu \,-\frac{1+\varepsilon}{Z} \gg 1,
\end{equation}
then $Z \ll \mathcal E$, so that $c^Z \ll_{\varepsilon} 1$ by \eqref{abc-pq}. 

Since $\mu \le 1/3+1/4=7/12$, if $Z \ge 3$, then, setting $\varepsilon:=1/12$, one finds that
\[
1-(1+\varepsilon)\mu \,-\frac{1+\varepsilon}{Z} 
\ge 1-\frac{7(1+\varepsilon)}{12} \,-\frac{1+\varepsilon}{3}
= \frac{1}{144}.
\]
Also, if $\mu \le 10/21$, then, setting $\varepsilon:=1/42$, one finds that
\[
1-(1+\varepsilon)\mu \,-\frac{1+\varepsilon}{Z}
\ge 1-\frac{10(1+\varepsilon)}{21} \,-\frac{1+\varepsilon}{2}
= \frac{1}{1764}.
\]
To sum up, if $Z \ge 3$ or $\mu \le 10/21$, then inequality \eqref{ineq-abc-epmuZ} holds, and $c^Z \ll 1$, $a,b,c \ll 1$.

Finally, we assume that $Z=2$ and $\mu>10/21$.
We shall finish the proof by showing that this leads to a contradiction. 
Since $z<Z$, and $\mu>10/21$ if and only if $\{X,Y\} \in \{\{3,4\},\{3,5\},\{3,6\}\}$, there is no loss of generality in assuming that we are in the case where
\[
a^2+b^2=c, \ \ a^3+b^Y=c^2
\]
with $Y \in \{4,5,6\}$.
If $Y=4$, then $a=a^3/a^2=(c^2-b^4)/(c-b^2)=c+b^2>c$, which immediately yields a contradiction to the first equation above. 
Further, reducing the above displayed equations modulo $a^2$ and $b^2$ implies that $b^{Y-4} \equiv 1 \pmod{a^2}$ and $a \equiv 1 \pmod{b^2}$, respectively.
If $Y \in \{5,6\}$, then these congruences together show that $b^{Y-4}>a^2>(b^2)^2=b^4$, so that $Y \ge 9$, a contradiction. 
\end{proof}

\begin{lem}\label{abc-minxy1-minXY1}
Assume that $abc$-conjecture is true.
If $\max\{a,b\}$ exceeds some constant depending only on $c,$ then $(x,y,X,Y)=(2,1,1,2)$ or $(1,2,2,1).$
\end{lem}

\begin{proof}
We may assume that $\max\{a,b\}>{\mathcal C}_{4},\mathcal C'$.
Lemmas \ref{c-general-minxy=1--minXY=1} and \ref{abc-xneX-yneY}\,(i) say that
\begin{gather*}
\min\{x,y\}=1, \ \min\{X,Y\}=1,\\
x \ne X, \ y \ne Y, \ z<Z,
\end{gather*}
respectively.
Therefore, there is no loss of generality in assuming that $x>1,y=1, X=1, Y>1$, that is,
\[
a^x+b=c^z, \quad a+b^Y=c^Z.
\]
We shall use inequality \eqref{abc-pq} with $\varepsilon=1/7$.
Since $p \le z/x, \, q \le Z/Y$ and $Z>z$, one has
\begin{align*}
\mathcal E &\ge Z-(z/x+Z/Y)(1+\varepsilon) \,-\,1-\varepsilon \\
&\ge \bigg(1-\biggr(\frac{1}{x}+\frac{1}{Y}\bigg)(1+\varepsilon) \bigg)\,z \,-1\,-\varepsilon.
\end{align*}
Suppose that $x>2$ or $Y>2$.
Since $1/x+1/Y \le 5/6$, it follows that
\[
\mathcal E \ge \biggr(\frac{1}{6}-\frac{5\varepsilon}{6} \bigg)\,z-1-\varepsilon=\frac{z}{21}-\frac{8}{7}.
\]
Thus $z \ll 1$ by inequality \eqref{abc-pq}.
This completes the proof.
\end{proof}

Lemma \ref{abc-minxy1-minXY1} together with \cite[Corollary 1]{MiyPin2} and Lemmas \ref{div}\,(ii) implies the following contribution to Step 1 of Problem \ref{prob-abc}.

\begin{thm}\label{c_abc}
Let $c$ be any fixed positive integer satisfying at least one of the following conditions\,{\rm :}
\begin{itemize}
\item[(i)] $c \not\equiv 0 \pmod{9}$ and $c$ has no prime factor congruent to $1$ modulo $3\,;$
\item[(ii)] $c$ is a prime.
\end{itemize}
Then, assuming the truth of abc-conjecture, $N(a,b,c) \le 1,$ except for only finitely many pairs of $a$ and $b.$
\end{thm}

Note that the first condition in the above theorem is equivalent to that $\varphi(c) \not\equiv 0 \pmod{3}$.

\begin{proof}[Proof of Theorem $\ref{c_abc}$]
By Lemma \ref{abc-minxy1-minXY1}, without loss of generality, we may assume that we have the following system of two equations: 
\[
a^2+b=c^z, \quad a+b^2=c^Z,
\]
where $E=e_{c}(a)=e_{c}{b}$ and $z \le Z$.
Observe that $\Delta=3$.

Since $E \mid \gcd(\varphi(c),\Delta)$, one sees under the condition (i) that $E=1$.
Then \cite[Corollary 1]{MiyPin2} completes the proof for (i).

Since it is clear that both $a^2+b, a+b^2$ are divisible by $c^z$, it follows that $a^2+b-(a+b^2) \equiv 0 \pmod{c^z}$, so that
\[
(a-b)(a+b-1) \equiv 0 \mod{c^z}.
\]
Since $a+b-1<a^2+b=c^z$, one obtains $\gcd(a-b,c^z)>1$.
Thus, if $c$ is a prime, then $a-b \equiv 0 \pmod{c}$, which immediately yields that $a \equiv -1 \pmod{c}$, so that $E=1$.
Then \cite[Corollary 1]{MiyPin2} completes the proof for condition (ii).
\end{proof}

\section{Open problems} \label{sec-op}

In this final section, we discuss further problems.

\subsection{Solving the system with the values of $c$ not being handled}\label{sec-op-1}

Here we shall be interested in considering Step 1 of Problem \ref{prob-abc} for the values of $c$ not being handled by the results in this paper and our previous work.
In this direction, not only Theorem \ref{c-7etc} but also Proposition \ref{result-MiyPin2} fails to handle the following values of $c$ (not perfect powers) as follows:
\begin{equation} \label{c-value-1}
c=10, 11, 14, 15, 19, 20, 21, 22, 23, 26, 28, 29, 30, 31, 33, 34, \dots
\end{equation}
For our aim for each case in \eqref{c-value-1}, it is enough to solve the system of 1st and 2nd equations with the exponents $x,y,X$ and $Y$ restricted as in the conclusions of Lemmas \ref{c-general-summary} or \ref{c-prime-summary}.
We leave open problems corresponding to some prime values of $c$.

\begin{prob}\label{prob-c11}
Let $x$ and $Y$ be positive integers with $2 \le x \le 3$ and $2 \le Y \le 3.74 \cdot 10^{11}$ such that $xY \equiv 1 \pmod{5}.$
Prove that there are only finitely many pairs $(a,b)$ of relatively prime positive integers greater than $1$ such that
\[
a^x+b=11^z, \ \ a+b^Y=11^Z
\]
for some positive integers $z$ and $Z$ with $z \le Z.$
\end{prob}

\begin{prob}\label{prob-c11}
Let $x$ and $Y$ be positive integers with $x \in \{2,4,5,7\}$ and $2 \le Y \le 1.3 \cdot 10^{13}$ such that $xY \equiv 1 \pmod{3}.$
Prove that there are only finitely many pairs $(a,b)$ of relatively prime positive integers greater than $1$ such that
\[
a^x+b=19^z, \ \ a+b^Y=19^Z
\]
for some positive integers $z$ and $Z$ with $z \le Z.$
\end{prob}

\begin{rem}\rm
For some composite value of $c$, if one succeeds in making the conclusion of Lemma \ref{c-general-xleEyleE} sharper, then it might be possible to work out Step 1 of Problem \ref{prob-abc} for such $c$. 
However, it seems that working out Step 2 further is difficult by the methods of this paper, because one has to rely on an explicit version of Proposition \ref{gp} (cf.~\cite[Theorem 2]{Bu}) in order to in particular consider how large the size of $\mathcal C_1$ in Lemma \ref{c-general-coprime} is (but, the so-called modular approach on ternary Diophantine equations (cf.~\cite[Ch.\,14]{Co}) might be helpful towards overcoming such a difficulty).
\end{rem}

Among the values of $c$ in \eqref{c-value-1}, even assuming $abc$-conjecture, we still fail in Theorem \ref{c_abc} to handle the following cases:
\begin{equation} \label{c-value-2}
c=14, 21, 26, 28, 35, 38, 39, 42, 52, 57, 62, 65, 70, 74, 76, 77, \dots
\end{equation}
For considering Problem \ref{prob-abc} for each case in \eqref{c-value-2}, we are by Lemma \ref{abc-minxy1-minXY1} led to consider the following problem:

\begin{prob}\label{prob-c14}
For each value of $c$ in \eqref{c-value-2}$,$ prove that there are only finitely many pairs $(a,b)$ of relatively prime positive integers greater than $1$ such that
\[
a^2+b=c^z, \ \ a+b^2=c^Z
\]
for some distinct positive integers $z$ and $Z.$
\end{prob}

Solving Problem \ref{prob-c14} only for a special value of $c$ (for instance, $c=14$) is a good result even if assuming the truth of $abc$-conjecture.\footnote{Problem \ref{prob-c14} was recently solved in \cite{LeMiy}.}

\subsection{Restricted rational approximations to certain algebraic irrationals}\label{sec-appro}

As seen in the proof of Proposition \ref{BuLu-newyork}, in order to make Theorem \ref{a-general-pillai_gcd1_ineff} (or each of Corollaries \ref{cor-a-prime-pillai} and \ref{cor-to-BuLu}) effective for each prime value of $a$, it is enough, for any positive integers $m$ and $i$ with $m>1$ and $i<m$ and for any positive number $\varepsilon$, to find a positive constant $\mathcal C$ such that the inequality
\begin{equation}\label{xi-ineq1}
\left|a^{\frac{i}{m}}-\frac{P}{Q}\right| >\frac{\mathcal C}{Q^{1+\varepsilon}}
\end{equation}
holds for any integer $P$ and for any positive integer $Q$ which equals a power of $a$.

As argued in Section \ref{sec-eff}, in order to make Theorem \ref{c-7etc} effective for each prime $c$ in \eqref{c-2r31prime}, it is enough to find positive constants $\epsilon$ and $\mathcal C$ with $\epsilon<0.6$ such that the inequality
\begin{equation}\label{xi-ineq2}
\left|\sqrt{c}-\frac{P}{Q}\right| >\frac{\mathcal C}{Q^{1+\epsilon}}
\end{equation}
holds for any integer $P$ and for any positive integer $Q$ which equals a power of $c$.

Closely related works to establishing inequalities \eqref{xi-ineq1} and \eqref{xi-ineq2} can be found in \cite{Beu,BaBe,Be_crelle_01}, but it seems that the methods therein are not enough to establish such results, except for \eqref{xi-ineq2} with $c=13$.\footnote{private communication with M. Bennett at Number Theory Conference 2022 at the University of Debrecen in Hungary}

\subsection*{Acknowledgements}
We would like to sincerely thank the anonymous referee for the thorough reading and for the helpful remarks and suggestions which substantially improved an earlier draft.
We are also grateful to Mihai Cipu and Reese Scott for their comments and remarks.

\end{document}